\documentclass[a4paper,12pt]{amsart}
\usepackage{hyperref}
\usepackage{geometry}
\usepackage[british]{babel}
\usepackage{wrapfig}
\usepackage[comma, square, numbers]{natbib}
\usepackage{yfonts}
\usepackage{comment}

\usepackage[utf8]{inputenc}
\usepackage{amssymb}
\usepackage[normalem]{ulem}
\usepackage{amsthm}
\usepackage{graphics}
\usepackage{amsmath}
\usepackage{amsthm}
\usepackage{dirtytalk}
\usepackage{amstext}
\usepackage{subfigure}
\usepackage{engrec}
\usepackage{floatflt}
\usepackage{rotating}
\usepackage[safe,extra]{tipa}
\usepackage[font={footnotesize}]{caption} 
\usepackage{framed}
\usepackage{listings}
\usepackage[titletoc,title]{appendix}

\newcommand{\mc}{\mathcal}
\newcommand{\mb}{\mathbb}

\newcommand{\R}{\mb R}

\newcommand{\N}{\mb N}
\newcommand{\Z}{\mb Z}

\newcommand{\T}{\mb T}

\newcommand{\E}{\mb E}

\newcommand{\eea}{\end{align}}

\renewcommand{\epsilon}{\varepsilon}
\renewcommand{\bar}{\overline}

\newcommand{\bo}{\boldsymbol}
\renewcommand{\phi}{\varphi}

\DeclareMathOperator{\Leb}{Leb}

\DeclareMathOperator{\Osc}{Osc}

\renewcommand\upsilon{\theta}
\usepackage[normalem]{ulem}

\newtheorem{theorem}{Theorem}[section]

\newtheorem{corollary}[theorem]{Corollary}

\newtheorem{lemma}[theorem]{Lemma}
\newtheorem{proposition}[theorem]{Proposition}
\theoremstyle{definition}
\newtheorem{definition}[theorem]{Definition}
\theoremstyle{definition}
\newtheorem{assumption}{Assumption}[section]

\newtheorem{remark}[theorem]{Remark}
\newtheorem{example}[theorem]{Example}

\newtheoremstyle{algorithm}
{4pt}
{4pt}
{}
{}
{}
{:}
{\newline}
{}

\newtheorem{algorithm}{Algorithm}

\newcommand{\balgorithm}{\begin{algorithm}\begin{framed}\ }
\newcommand{\ealgorithm}{\end{framed}\end{algorithm}}

\newcommand{\pippo}{\begin{code} \  \begin{lstlisting}[frame=single]}
\newcommand{\pappo}{\end{lstlisting} \end{code}}

\newcommand{\bd}{\begin{definition}}
\newcommand{\ed}{\end{definition}}

\newcommand{\bt}{\begin{theorem}}
\newcommand{\et}{\end{theorem}}
\newcommand{\bp}{\begin{proposition}}
\newcommand{\ep}{\end{proposition}}
\newcommand{\bc}{\begin{corollary}}
\newcommand{\ec}{\end{corollary}} 
\newcommand{\bl}{\begin{lemma}}
\newcommand{\el}{\end{lemma}}
\newcommand{\br}{\begin{remark}}

\newcommand{\er}{\end{remark}}

\DeclareMathOperator{\Id}{Id}

\usepackage{enumitem}
\title[Uniformly Expanding Coupled Maps]{Uniformly Expanding Coupled Maps: Self-Consistent Transfer Operators and Propagation of Chaos}

\author{Samuel Ekstroem}
\address{Samuel Ekstroem: Department of Mathematics, King's College London, London, United Kingdom}
\email{samuel.ekstroem@kcl.ac.uk}
\author{Matteo Tanzi } 
\address{Matteo Tanzi: Department of Mathematics, Courant Institute of Mathematical Sciences, New York University, New York, USA}
\address{Matteo Tanzi: Laboratoire de Probabilit\'es, Statistique et Mod\'elisation (LPSM), UFR de Math\'ematiques, Universit\'e Paris Cit\'e, Paris, France}
\address{Matteo Tanzi: Department of Mathematics, King's College London, London, United Kingdom}
\email{matteo.tanzi@kcl.ac.uk}
\author{Giuseppe Tenaglia}
\address{Giuseppe Tenaglia: Department of Mathematics, Imperial College London, London, United Kingdom}
\email{giuseppe.tenaglia20@imperial.ac.uk}

\begin{document}

\begin{abstract}
We prove quantitative propagation of chaos in total
variation, uniformly in time, for systems of $N$ deterministic
uniformly expanding heterogeneously coupled maps. The local dynamics and interactions
may vary between coordinates, so exchangeability is not required.
Under assumptions (independent of the dimension) on expansion, regularity, and weak-coupling, we construct self-consistent transfer operators whose
product evolutions approximate every fixed finite marginal. For
 a class of weakly dependent initial
laws (containing products), the error is $O(N^{-1})$ on fixed time
intervals and $O(\log N/N)$ uniformly in time. The finite-time rate is
optimal within this class. An explicit family exhibits error of order
$\sqrt{\log N}/N$ at times of order $\log N$, proving that a uniform
$O(N^{-1})$ bound is impossible in general; the optimal correction
remains open. The proof is based on ergodic theory: expansion preserves
classes of regular, weakly dependent densities, while transfer-operator
contraction controls error accumulation. These classes yield
concentration estimates and allow direct comparison of densities,
giving strong propagation of chaos without external noise. Under
additional assumptions, we also obtain Gaussian approximation for
empirical observables on fixed time intervals and invariant
neighbourhoods of stable fixed
points of the self-consistent operator.
\end{abstract}
\maketitle
%\tableofcontents
\section{Introduction}

Globally coupled maps describe collections of deterministic dynamical
systems that evolve in discrete time. Each component follows a local
map and is influenced by the states of the other components. In the
mean-field regime, the aggregate interaction is normalised by the
number $N$ of components. Since the early studies of these systems,
numerical and analytical work has revealed interesting behaviours ranging from
synchronisation and clustering to collective chaos
\cite{kaneko1990clustering,pikovsky1994globally,kaneko1990globally,
just1995globally,ershov1995mean}. A central question is how to describe
their statistical evolution when $N$ is large, and for how long a
mean-field description remains accurate.

{The central contribution of this paper is an
ergodic-theoretic framework that turns deterministic chaos into
quantitative, strong propagation of chaos\footnote{Note that the first occurrence of the word “chaos” in this phrase refers to sensitive dependence on initial conditions.}. It yields the optimal
$N^{-1}$ scale on fixed time intervals and an approximation valid for
all times with only a logarithmic loss. An explicit example shows
that some loss is necessary. Thus we establish both the accuracy
and a limitation of the mean-field description for chaotic coupled
maps.\par}

{Propagation of chaos arose in kinetic theory and statistical mechanics: it describes the persistence under the dynamics of asymptotic independence among any fixed finite collection of coordinates as the number of interacting units  tends to infinity. Foundational results
concern collision models, nonlinear diffusions, and deterministic
Vlasov equations
\cite{kac1956foundations,mckean1966class,braun1977vlasov,dobrushin1979vlasov};
general frameworks and surveys are given in
\cite{sznitman1991topics,chaintron2022models,chaintron2022applications}.\par}

{Central questions concern singular interactions, strong
convergence, and quantitative control over long times. The literature on the topic is vast and ranges from
mean-field limits for singular Vlasov forces \cite{hauray2015particle},
to strong propagation of chaos via entropy estimates for bounded-force
Vlasov systems and diffusions with singular interactions
\cite{jabin2016mean,jabin2018quantitative}, sharp marginal rates for
diffusions, uniform in time under additional stability assumptions
\cite{lacker2023hierarchies,lacker2023sharp}, and uniform-in-time
estimates for collision models \cite{mischler2013kac}.\par}

{For heterogeneous systems, continuum and kinetic
descriptions retain node and network information. Results cover
Kuramoto and other ODEs on graphs
\cite{chiba2019mean,medvedev2019continuum,ayi2021mean,ayi2024large},
diffusions on random networks
\cite{delattre2016note,coppini2020law,oliveira2019interacting}, and
graphon and extended-graphon limits for non-exchangeable systems
\cite{bayraktar2023graphon,jabin2025nonexchangeable}. Quantitative
estimates also compare finite marginals directly with products of
node-dependent laws \cite{lacker2024nonexchangeable}.\par}

{For coupled maps, ergodic theory uses
\emph{self-consistent transfer operators} (STOs): the current distribution of the states
determines the average maps that transport those distributions forward. These operators
are discrete-time analogues of the Vlasov equation. Transfer-operator
methods give invariant measures, stability, and linear response
\cite{balint2018synchronization,galatolo2022self,keller2000ergodic,
selley2016mean,selley2021linear}. Propagation of chaos is known through
weak convergence of empirical measures \cite{bardet2009stochastically}
and relative-entropy estimates, together with Gaussian fluctuations
\cite{cheng2026propagation}; uniform-in-time results also hold for maps
with additive noise \cite{tenaglia2026uniform}. Here long-time control
comes from deterministic expansion and transfer-operator contraction.\par}

{We prove \emph{uniform-in-time propagation of chaos in
total variation} for uniformly expanding coupled maps. The local maps
and interactions are $C^3$, expansion is sufficiently strong relative
to distortion, and coupling is sufficiently small, with conditions
independent of $N$. Total variation controls all bounded measurable
observables of any fixed collection of coordinates. Starting from
regular product densities, these coordinates remain approximately
independent for all times; the results also cover a class of weakly
dependent initial densities. Randomness enters only through the
initial conditions.\par}

{The framework accommodates heterogeneous local maps and
interactions without exchangeability. We construct an STO on a
family of potentially different marginal laws, producing an
$N$-dependent product evolution that shadows the coupled system.
The estimates hold uniformly over the selected coordinates and
require no convergence assumption on the interaction networks.
This gives a quantitative statistical description of each finite
system, even when there is no single common limiting marginal.\par}

Our estimates give explicit rates. Let $\mu_t$ be the law of the
$N$-coordinate system at time $t$, let $\nu_{i,t}$ be the $i$-th
marginal of the STO evolution started from $\mu_0$, and let $\Pi_S$
denote projection onto the coordinates in a set
$S\subseteq\{1,\ldots,N\}$ of fixed cardinality $k$. We prove
\[
\max_{0\le t\le T}
\left\|\Pi_S\mu_t-\bigotimes_{i\in S}\nu_{i,t}\right\|_{\mathrm{TV}}
\le\frac{C_{T,k}}N,
\qquad
\sup_{t\ge0}
\left\|\Pi_S\mu_t-\bigotimes_{i\in S}\nu_{i,t}\right\|_{\mathrm{TV}}
\le C_k\frac{\log(2N)}N.
\]
Here $T$ is any fixed nonnegative integer, and the constants are
independent of $N$ and the choice of $S$. The finite-time bound holds
under the hypotheses of Theorem~\ref{Thm:Main}; the uniform bound
requires the additional smallness of the coupling in
Corollary~\ref{Cor:PropagationofChaos}. That corollary also gives the
explicit dependence on $k$.

{For fixed $k$, the finite-time rate $O(N^{-1})$ matches
the order known to be optimal for suitable stochastic mean-field
diffusions: relative-entropy estimates of order $(k/N)^2$ give total
variation bounds of order $k/N$, with matching fixed-marginal scaling
in Gaussian examples \cite{lacker2023hierarchies}. Under additional
assumptions, that order is also attained uniformly in time for
diffusions \cite{lacker2023sharp}. Our estimates achieve the same
finite-time order through the chaotic dynamics of deterministic
maps. This comparison concerns the power of $N$ for fixed $k$;
we do not claim optimal dependence on $k$.\par}

{The finite-time scale is also optimal within our own
class, and the long-time loss cannot be eliminated throughout this
class. In Example~\ref{Ex:RateObstruction}, a family of directed star
networks satisfies all expansion, regularity, and coupling assumptions
uniformly in $N$. The exact first-iterate computation in
Section~\ref{Sec:RateObstruction} gives a one-coordinate error of
order $N^{-1}$, already proving finite-time optimality. At a time of
order $\log N$, the same family's hub error has order
$\sqrt{\log N}/N$. Consequently, no estimate $C/N$ with $C$
independent of both $N$ and time can hold over the full admissible
class. This identifies a genuine distinction between finite-time
and uniform-in-time accuracy. The remaining gap between
$\sqrt{\log N}$ and $\log N$ is open: our upper logarithmic factor
is not known to be optimal.\par}

{Ergodic theory is the basis of the proof. We work with
classes of \emph{quasi-product measures}, which contain regular product
densities and allow weak dependence between coordinates. They control
the regularity of one-coordinate conditional densities and their weak
sensitivity to changes in any one other coordinate. Uniform expansion
and distortion estimates ensure that both the coupled dynamics and
the mean-field evolution preserve these classes, with bounds
independent of $N$. This provides a common setting for concentration
estimates and direct comparison of densities throughout the evolution.
Within this setting, cancellation of the leading interaction
fluctuations yields the $O(N^{-1})$ one-step total variation error.\par}

{The same ergodic-theoretic estimates also explain the
long-time result. The one-step error is $O(N^{-1/2})$ in the
derivative norm, and transfer-operator contraction makes old errors
decay. For sufficiently small coupling, this contraction also controls
feedback through the mean field. Combining the $O(N^{-1})$ total
variation size of each new error with its decay in the derivative
norm gives the uniform $O(\log N/N)$ bound. Thus the regularity and
loss of memory supplied by chaotic expansion sustain a strong
approximation by independent marginal laws over arbitrarily long times.\par}

As a further consequence, under the stability assumptions of
Corollary~\ref{Cor:FixedPointofOperators}, the coupled dynamics
preserves neighbourhoods of a fixed point of the STO with radii
$O(N^{-1/2})$ in the marginal derivative pseudometric $d_N$ defined
in Section~\ref{Sec:Results}. We also obtain Gaussian approximation
results for empirical observables under the assumptions stated there.

The paper is organised as follows. Section~\ref{Sec:Setup} defines
the coupled systems and their STOs, and Section~\ref{Sec:Results}
states the main results. Section~\ref{Sec:ProofMain} proves
Theorem~\ref{Thm:Main}, while Section~\ref{Sec:ProofMainResBig} proves
the propagation-of-chaos and stability corollaries.
Section~\ref{Sec:RateObstruction} verifies
Example~\ref{Ex:RateObstruction}, and
Section~\ref{Sec:GaussianFluctuations} proves the Gaussian fluctuation
results. The appendices establish the concentration inequality,
the transfer-operator and marginalisation estimates, and invariance
of the quasi-product class.

 \smallskip
{\bf Acknowledgments:} M.~Tanzi acknowledges support from
the Marie Sk\l odowska-Curie Actions project ``Ergodic theory for complex
systems: a rigorous study of dynamics on heterogeneous networks''
(grant agreement no.~843880), the EPSRC--FAPESP grant no.~2023/13706,
and EPSRC grant no.~UKRI1021. G.T. was supported by the Engineering and Physical Sciences Research Council [grant number EP/Z533658/1].

\smallskip
{\bf AI Disclosure:} The authors used ChatGPT (OpenAI; various Astra and Sol models) and Claude (Anthropic; Opus 5.5) to assist with developing and checking mathematical arguments and examples, identifying and correcting gaps in proofs, searching the literature, and revising the exposition and LaTeX source. This included assistance with the construction and proof of Example~\ref{Ex:RateObstruction}. All AI-assisted material retained in the manuscript was independently checked and revised by the authors. The authors take full responsibility for the manuscript, including the correctness of its mathematical statements, proofs, and references.

\subsection{Notation}
Let  $\T=\R/\Z$ be the 1D torus and for every $i,j\in\N$ with $i\le j$  define the set of indices $[i,j]:=\{i, i+1,...,j\}$.

We write $\partial_j v=\partial v/\partial x_j$ and
$\partial_{jk}v=\partial^2v/(\partial x_j\partial x_k)$.
Subscripts on functions denote components or labels, never derivatives.
For functions of one variable, primes denote ordinary derivatives.
We write $\lambda$ for normalized Lebesgue measure on $\T$ and use
$\lambda^{\otimes d}$ on $\T^d$; the symbols $dx$ and $\Leb$ refer to
these measures. Unspecified $L^p$ norms are taken with respect to them.
For a real function $v\in C^r(\T^d)$, set
$\|v\|_{C^r}:=\max_{|\alpha|\le r}\|\partial^\alpha v\|_\infty$.
Time indices take values in $\mathbb N_0=\{0,1,\ldots\}$.
For a measure $\mu$ and an integrable function $v$, write
$\mu(v)=\int v\,d\mu$. If $p$ is a density, $T_*p$ denotes the density
of $T_*(p\,d\lambda^{\otimes d})$ whenever this pushforward is
absolutely continuous. Likewise, the total variation of a signed
density means the total variation of its associated measure.

\section{Setup}\label{Sec:Setup}
 We are going to consider a system of $N\in \N$ coupled maps where each map is described by a variable $x_i\in \T$. For every $i\in[1,N]$, let $f_i:\T\rightarrow \T$ be the  $i$-th \emph{uncoupled} map, and for every $i,j\in[1,N]$ let $h_{ij}:\T\times \T \rightarrow \R$ prescribe the shape of the interactions between $i$-th and $j$-th maps.  The regularity of these functions will be prescribed later. Given $\widetilde f_i:\R\rightarrow \R$ and $\widetilde h_{ij}:\R\times\R\rightarrow \R$  lifts of $f_i$ and $h_{ij}$, define  $\bo F:=(F_1,...,F_N):\T^N\rightarrow \T^N$ as
 \begin{equation}\label{eq:defF}
 F_i(x_1,...,x_N):= \widetilde f_i(x_i)+\frac{1}{N}\sum_{j=1}^N\widetilde h_{ij}(x_i,x_j)\mod 1
 \end{equation}
 which is the map for the evolution equations of the system of globally coupled maps.
 We use tildes for lifts; for $h_{ij}$ this means its periodic extension.
We suppress tildes in derivatives and in formulas understood modulo one.
Differences and interpolations of circle maps always refer to fixed lifts
of the same degree, and their $C^r$ norms mean the norms of these
periodic differences. Bars are reserved for averages.

Throughout the paper we assume $h_{ii}=0$. A diagonal
interaction may instead be absorbed into $f_i$, provided
Assumption~\ref{Ass:CondF} holds for the resulting maps.

\subsection{Regularity assumptions on \texorpdfstring{$\bo F$}{F}} In our results we study maps as in \eqref{eq:defF} satisfying the following assumption.
 \begin{assumption}\label{Ass:CondF} Given $(\kappa,K,E)\in (\R^+_0)^3$ with $\kappa-E>1$,  $f_i,\,h_{ij}$ are in $C^3$ and satisfy
\begin{equation}\label{Eq:BOundh}
\|h_{ij}\|_{C^3} \le E
\end{equation}
\begin{equation}\label{Eq:CondDHEntries}
\|f_i''\|_\infty,\ \|f_i'''\|_\infty\le K,
\qquad \inf_{x\in\T}|f_i'(x)|>\kappa
 \end{equation}
 for all $j,i\in[1,N]$ and $j\neq i$.
\end{assumption}
The above assumption implies that $\bo F$ is a local diffeomorphism\footnote{Periodicity and the second-derivative bounds give $\|\partial_1h_{ij}\|_\infty,\|\partial_2h_{ij}\|_\infty\le E/4$: each first derivative has mean zero in its corresponding variable, and the integral of geodesic distance from a point on the unit circle is $1/4$. Consequently, $|\partial_iF_i|-\sum_{j\ne i}|\partial_jF_i|\ge\kappa-E/2>1$, so $D\bo F$ is invertible by strict diagonal dominance.}. Notice that condition  \eqref{Eq:BOundh} implies that the influence of the $j$-th coordinate  on the evolution of the $i$-th coordinate is of order at most $N^{-1}$. Furthermore, given $\bo F$ as in \eqref{eq:defF} satisfying Assumption  \ref{Ass:CondF},  for every fixed $i\in[1,N]$ and $\hat{\bo x}_i\in \T^{N-1}$, $F_i(\cdot;\,\hat{\bo x}_i):\T\rightarrow \T$\footnote{A ``$\cdot$" in place of $x_i$ emphasizes that $x_i$ is thought as the free variable for this function, while $\hat{\bo x}_i$ is considered a fixed parameter.} is a $C^3$ uniformly expanding map with expansion lower bounded by $\kappa-E$ and distortion upper bounded by
\[
\mc D:=\frac{K+E}{(\kappa-E)^2}.
\]

\subsection{Self-consistent transfer operator} \label{Sec:HeuristTransferOperator}For $\bo x=(x_1,...,x_N)\in \T^N$ and $i\in[1,N]$ we adopt the notation  \[\hat {\bo x}_i=(x_1,...,x_{i-1},x_{i+1},...,x_N)\in \T^{N-1}\mbox{ and }\bo x=(x_i;\,\hat{\bo x}_i).\] 
For $S\subseteq\{1,\ldots,N\}$, let
\[
\pi_S(x)=(x_i)_{i\in S},\qquad \Pi_S:=(\pi_S)_*,
\]
with coordinates listed in increasing order. Write $\pi_i=\pi_{\{i\}}$,
$\Pi_i=\Pi_{\{i\}}$, and
\[
\widehat\pi_i=\pi_{\{1,\ldots,N\}\setminus\{i\}},\qquad
\widehat\Pi_i=(\widehat\pi_i)_*.
\]
Thus $\Pi_i\mu$ and $\widehat\Pi_i\mu$ are the corresponding marginals.
We denote the conditional law of coordinate $i$ given the remaining
coordinates $z$ by $\mu_{i\mid-i}(\cdot\mid z)$, and its density,
when it exists, by $p_{i\mid-i}(\cdot\mid z)$.
When the conditioning set is smaller, we specify it explicitly.

\begin{definition}[Mean-Field Approximation]\label{meanfieldapprox}
Given $\bo F:\T^N\rightarrow \T^N$ and $\mu\in \mc M_1(\T^N)$, for every $i\in [1,N]$ define $F_{\mu,i}:\T\rightarrow \T$ 
\[
F_{\mu,i}(x_i):=\widetilde f_i(x_i)+\frac1N\sum_{j\ne i}
\int_\T h_{ij}(x_i,y)\,d\Pi_j\mu(y)\mod 1
\]
and  the product map $\bo F_\mu:=(F_{\mu,1},...,F_{\mu,N})$ on $\T^N$, which we call the \emph{mean-field approximation} of $\bo F$ with respect to $\mu$.
\end{definition}
The rationale for introducing the above mean-field maps is the following: Consider $\mu=\mu_1\otimes...\otimes \mu_N$ a product measures on $\T^N$ with $\mu_1,...,\mu_N\in \mc M_1(\T)$; then if $\bo F$ satisfies Assumption \ref{Ass:CondF}, the dependence of $F_i$ on each coordinate $j\neq i$ is of order $N^{-1}$ and classical results on concentration of measure\footnote{E.g. McDiarmid's inequality \cite{mcdiarmid1998concentration}.} imply that the map $F_i(\cdot;\,\hat{\bo x}_i):\T\rightarrow \T$ is ``close" to $F_{\mu,i}$ for $\hat{\bo x}_i$ in a subset of $\T^{N-1}$ whose complement has  $\widehat\Pi_i\mu$-measure exponentially small in the system's dimension $N$. This suggests the heuristics that there is $\mc G\subset \T^N$ with $\mu(\mc G)\approx 1$ where $\bo F$ can be approximated by the product map $\bo F_\mu$. In practical terms, this means that for $N$ finite, but very large, if we were to investigate the evolution under $\bo F$ with finite precision and drew an initial condition at random with respect to the measure $\mu$, with high probability, we would expect  the evolution of the $i$-th coordinate to be indistinguishable from the map $F_{\mu,i}$. In turn, the measure $\bo F_*\mu$ is expected to be approximately $(\bo F_{\mu})_*\mu$. 

Guided by these heuristic arguments, we give the following definition.
\begin{definition}[Self-Consistent Transfer Operator] Given $\bo F:\T^N\rightarrow \T^N$, let $\bo{\mc F}:\mc M_1(\T^N)\rightarrow \mc M_1(\T^N)$ be defined as
\[
\bo{\mc F}\,\mu:= (\bo F_\mu)_*\mu.
\]
\end{definition} 

In general, $\bo{\mc F}$ is nonlinear, as the linear operator $(\bo F_\mu)_*$ depends on the measure it is applied to. $\bo{\mc F}$ has been given the name of \emph{self-consistent transfer operator} (STO). Notice that if the uncoupled maps $f_i$ and the interaction functions $h_{ij}$ are all identical and $\bo{ \mc F}$ is restricted to product measures having identical factors, i.e. the system has full permutation symmetry, then we recover, up to the factor $(N-1)/N$ in front of the interaction term (due to $h_{ii}=0$), the standard definition of self-consistent operator that previously appeared in the literature; therefore this definition is a generalization of the STO to cases without full permutation symmetry. %\footnote{Another subtle difference between the definition of STO given here and the one given in the literature, is that here $\bo {\mc F}$ depends on $N$.}.  

 The goal of this paper is to study to which extent $\bo F_*$ can be approximated by $\bo{\mc F}$.

\section{Results}\label{Sec:Results}
Roughly speaking, we find sufficient conditions on $\bo F$  so that $\bo F_*$ keeps invariant a class of measures close to product that satisfy  concentration inequalities and whose evolution is well approximated by the self-consistent transfer operator $\bo{\mc F}^t$, in a sense that will be made precise below.

\subsection{Main results}

We consider the following family of probability measures.

\begin{definition}[The invariant class]\label{Def:InvariantClass}
For $a,M>0$, $\gamma\ge0$, and $N\ge2$, define
\begin{equation}\label{Eq:InvariantClass}
\mc Y_N(a,M,\gamma):=\left\{\rho\in C^2(\T^N):
\begin{array}{l}
\rho>0,\quad \int\rho=1,\\
\displaystyle\max_i\|\partial_i\log\rho\|_\infty\le a,\\
\displaystyle\max_i\|\partial_{ii}\log\rho\|_\infty\le M,\\
\displaystyle\max_{i\ne j}\|\partial_{ij}\log\rho\|_\infty
\le\gamma/N
\end{array}\right\}.
\end{equation}
\end{definition}
For an absolutely continuous probability measure $\mu$, the notation
$\mu\in\mathcal Y_N(a,M,\gamma)$ means
$d\mu/d\lambda^{\otimes N}\in\mathcal Y_N(a,M,\gamma)$.

Note that this class contains all products of positive $C^2$
probability densities $p_i$ satisfying
\[
\|(\log p_i)'\|_\infty \le a,
\qquad
\|(\log p_i)''\|_\infty \le M.
\]
More generally, for densities in this class, the dependence between
distinct coordinates vanishes as $N\to\infty$, motivating the term
\emph{quasi-product measures}. Moreover, as we state below, when $\gamma$ is small enough,
measures with densities in $\mathcal Y_N(a,M,\gamma)$ satisfy
concentration estimates analogous to those for product measures.

\begin{definition}\label{def:oscillation}
Given a function $\psi:\T^k\rightarrow \R$, define its oscillation with respect to the $j$-th coordinate as
\[
\Osc_j\psi:=\sup\left\{|\psi(x_j;\,\hat{\bo x}_j)-\psi(x_j';\,\hat{\bo x}_j)|:\,x_j,x_j'\in\T,\,\hat{\bo x}_j\in \T^{k-1}\right\}.
\]
For $k\le N$ and $\alpha\ge 0$, define 
\[
\mc O_{\alpha}(\T^k):=\{\psi:\T^{k}\rightarrow \R:\, \forall j\in[1,k]\,\,\Osc_j(\psi)\le \alpha,\, \psi\in L^1(\T^k,\,\Leb)\}.
\]
\end{definition}
For a bounded real function $g$ on a set $A$, write
$\operatorname{osc}_{A}g=\sup_{x,y\in A}|g(x)-g(y)|$;
the subscript is omitted when the domain is clear.
The following result holds.
\begin{proposition}[Concentration properties]
\label{Thm:ConcMeasure}
Fix $a,M>0$ and $\gamma\in (0,1).$ Then, for every $N\ge2$, if
$\mu\in\mathcal M_1(\T^N)$ has density in
$\mc Y_N(a,M,\gamma)$, we have
\[
\mu\left(
    \left|g-\mb E_\mu[g]\right|\ge\epsilon
\right)
\le
2\exp\left[
    -\frac{
        \left(1-\frac{\gamma}{16}\frac{N-1}{N}\right)\epsilon^2
    }{
        \displaystyle\sum_{i=1}^N
        \left(\operatorname{Osc}_i g\right)^2
    }
\right]
\qquad\text{for every }\epsilon\ge0,
\]
for every $\alpha>0$ and every nonconstant
$g\in\mc O_\alpha(\T^N)$. In particular, for every $g\in\mc O_{N^{-1}}(\T^N)$,
\[
\mu\left(
    \left|g-\mb E_\mu[g]\right|\ge\epsilon
\right)
\le
2\exp\left[-\left(1-\frac{\gamma}{16}\right)N\epsilon^2\right]
\qquad\text{for every }\epsilon\ge0.
\]
\end{proposition}
The proof of this result is in Appendix~\ref{Sec:Quasi-Product-Chatterjee}.

Our first main result establishes two properties of the coupled dynamics
and its associated self-consistent transfer operator
$\bo{\mc F}$. First, when the expansion is sufficiently strong compared
to the distortion and the interaction, both $\bo F_*$ and $\bo{\mc F}$
preserve a class $\mc Y_N(a,M,\gamma)$ for which the above concentration
inequality holds. Second, on this class, one step of the coupled dynamics is close to one application of the self-consistent transfer operator. 

More precisely, for densities $p,q$ on $\T$, let us define the metric
\begin{equation}\label{metric}
d_{\partial}(p,q):=\|p'-q'\|_{1},
\end{equation}
and the pseudometric on measures on $\T^N$  
\begin{equation}\label{pseudometric}
d_N(\mu,\nu):=\max_{1\le i\le N} d_{\partial }(\Pi_i\mu,\Pi_i\nu).
\end{equation}
For absolutely continuous measures $\alpha,\beta$ on $\T$, we use
$
d_\partial(\alpha,\beta)
:=d_\partial(d\alpha/d\lambda,d\beta/d\lambda).
$

\begin{theorem}\label{Thm:Main}
Fix $K\ge 0$. For every sufficiently large $\kappa>0$ (depending on $K$)
and every sufficiently small $E>0$ (depending on $K$ and $\kappa$),
there exist $a>0$, $M\ge 1$, and
$\gamma\in(0,1/16)$, with $\gamma=O(E)$, such that the
following holds for all $N\ge 2$. If $\bo F$ satisfies Assumption~\ref{Ass:CondF} with
parameters $K$, $E$, and $\kappa$, then
\begin{equation}\label{invcone}
\bo F_*\mc Y_N\subset\mc Y_N,
\qquad \bo{\mc F}(\mc Y_N)\subset\mc Y_N,
\end{equation}
where $\bo{\mc F}$ is the self-consistent transfer operator associated to $\bo F,$ and
\begin{equation}\label{distance}
d_N(\bo F_*\mu,\bo{\mc F}\mu)
\le \bar C N^{-1/2}
\qquad(\mu\in\mc Y_N),
\end{equation}
for some constant $\bar C>0$ independent of $N$ and $\mu$.
Moreover, enlarging $\bar C$ if necessary,
\begin{equation}\label{Eq:MainOneStepTV}
\max_i\|\Pi_i\bo F_*\mu-\Pi_i\bo{\mc F}\mu\|_{\mathrm{TV}}
\le \bar C/N\qquad(\mu\in\mc Y_N).
\end{equation}
\end{theorem}

Let us recall that the total variation is defined as
\begin{equation}\label{totalvar}
\|\sigma\|_{\mathrm{TV}}
:=\sup_{\|\varphi\|_\infty\le1}
\left|\int\varphi\,d\sigma\right|.
\end{equation}
For probability measures $\alpha,\beta$, we also use
\[
d_{\mathrm{TV}}(\alpha,\beta)
:=\tfrac12\|\alpha-\beta\|_{\mathrm{TV}}.
\]
If $p,q$ are their densities, then
$\|\alpha-\beta\|_{\mathrm{TV}}=\|p-q\|_1$.
The same conventions apply to probability mass functions on finite sets.

A corollary of Theorem \ref{Thm:Main} is the  following \emph{uniform-in-time propagation of chaos}.

\begin{corollary}[Propagation of chaos]\label{Cor:PropagationofChaos}
Under the hypotheses of Theorem~\ref{Thm:Main}, with the coupling
sufficiently small, there is $C>0$ such that, for every $N\ge2$,
$\mu\in\mc Y_N$, and $S\subseteq\{1,\ldots,N\}$ of cardinality
$1\le k\le N$,
\begin{equation}\label{Eq:ChaosUniformTV}
\sup_{t\ge0}
\left\|\Pi_S\bo F_*^t\mu
-\bigotimes_{i\in S}\Pi_i\bo{\mc F}^{\,t}\mu\right\|_{\mathrm{TV}}
\le C\frac{k(k-1)+k\log(2N)}{N}.
\end{equation}
Moreover, under the hypotheses of Theorem~\ref{Thm:Main} alone (without
further smallness of the coupling), for every integer $T\ge0$ there is
$C_T>0$ such that, for all $N$, $\mu$ and $S$ as above,
\begin{equation}\label{Eq:ChaosFixedTimeTV}
\max_{0\le t\le T}
\left\|\Pi_S\bo F_*^t\mu
-\bigotimes_{i\in S}\Pi_i\bo{\mc F}^{\,t}\mu\right\|_{\mathrm{TV}}
\le C_Tk^2/N.
\end{equation}
\end{corollary}

The following example shows that the logarithmic factor in
\eqref{Eq:ChaosUniformTV} cannot be removed uniformly over the
systems covered by Corollary~\ref{Cor:PropagationofChaos}. More
precisely, even for $k=1$ and product initial data, there is no
constant $\hat C,$ depending only on the fixed bounds $\kappa$, $K$, and $E$,
such that \eqref{Eq:ChaosUniformTV}  holds with $C=\hat C$ and without the logarithmic term.

\begin{example}[Directed star network]
\label{Ex:RateObstruction}
Fix $K>0$ and choose $\kappa$ and $E>0$ so that the hypotheses of
Corollary~\ref{Cor:PropagationofChaos} are satisfied. Choose an odd integer $q\ge5$ sufficiently large, and then choose
$\alpha,\varepsilon>0$ sufficiently small. For each $\ell\ge1$, set
\[
N_\ell=q^{2\ell},
\qquad m_\ell=q^\ell,
\qquad
B_\ell(y)
=y+\frac{\alpha}{2\pi m_\ell^3}\sin(2\pi m_\ell y),
\qquad A_\ell=B_\ell^{-1}.
\]
Define $\bo F^{(N_\ell)}:\T^{N_\ell}\to\T^{N_\ell}$ by
\begin{equation}\label{Eq:ObstructionArray}
\begin{aligned}
x_1(t+1)
&=A_\ell(qx_1(t))
+\frac{\varepsilon}{N_\ell}\sum_{j=2}^{N_\ell}
\left[
\sin(2\pi x_j(t))
-q^{-1}\sin(2\pi qx_j(t))
\right]\,\,\bmod 1,\\
x_j(t+1)&=qx_j(t)\,\, \bmod 1\qquad 2\le j\le N_\ell.
\end{aligned}
\end{equation}

The parameters $q,\alpha,\varepsilon$ can be chosen so that, for
every $\ell\ge1$, the map $\bo F^{(N_\ell)}$ satisfies the
assumptions of Corollary~\ref{Cor:PropagationofChaos} with the
same bounds $\kappa$, $K$, and $E$. Let $\bo{\mc F}_{N_\ell}$
denote its self-consistent transfer operator and define
\[
\mu_t^{(N_\ell)}
=\bigl(\bo F^{(N_\ell)}_*\bigr)^t\lambda^{\otimes N_\ell},
\qquad
\nu_t^{(N_\ell)}
=\bo{\mc F}_{N_\ell}^{\,t}\lambda^{\otimes N_\ell}.
\]
Then there exist constants $c>0$ and $\ell_0\ge1$, independent of
$\ell$, such that, for every $\ell\ge\ell_0$, at time
$t_\ell=\ell+1$,
\begin{equation}\label{Eq:RateObstruction}
\left\|\Pi_1\mu_{t_\ell}^{(N_\ell)}
-\Pi_1\nu_{t_\ell}^{(N_\ell)}\right\|_{\mathrm{TV}}
\ge c\frac{\sqrt{\log N_\ell}}{N_\ell}.
\end{equation}
In particular, no bound of order $N^{-1}$ can hold uniformly in time
and over all systems satisfying the same bounds $\kappa$, $K$, and $E$
\end{example}
The proof of \eqref{Eq:RateObstruction}  in the Example ~\ref{Ex:RateObstruction} is given in Appendix~\ref{Sec:RateObstruction}.

The next corollary shows that if $\bo{\mc F}$ has a fixed
point in $\mc Y_N$ that is stable in the  pseudometric $d_N$, then  $\bo{F}_*$ preserves a $d_N$-neighbourhood of such fixed point whose size is proportional to $N^{-\frac{1}{2}}$.
\begin{corollary}[A neighbourhood of a stable fixed point]
\label{Cor:FixedPointofOperators}
Under the hypotheses of Theorem~\ref{Thm:Main}, suppose
that for every sufficiently large $N$ there is a fixed point
$\bar\mu_N\in\mc Y_N$ of $\bo{\mc F}$. Set
$
B_r^N(\bar\mu_N)
:=\{\mu\in\mc Y_N:d_N(\mu,\bar\mu_N)<r\}.
$
Assume that there are $\delta>0$ and $\lambda_{\mathrm{fp}}\in[0,1)$,
independent of $N$, such that
\[
d_N(\bo{\mc F}\mu,\bar\mu_N)
\le\lambda_{\mathrm{fp}} d_N(\mu,\bar\mu_N)
\qquad(\mu\in B_\delta^N(\bar\mu_N)).
\]
Then there are $c>0$ and $N_0$ such that
\[
\bo F_*\bigl(B_{cN^{-1/2}}^N(\bar\mu_N)\bigr)
\subset B_{cN^{-1/2}}^N(\bar\mu_N)
\qquad(N\ge N_0).
\]
\end{corollary}

The proofs of Corollaries~\ref{Cor:PropagationofChaos} and~\ref{Cor:FixedPointofOperators} are given in Section~\ref{Sec:ProofMainResBig}, using concentration from Appendix~\ref{Sec:Quasi-Product-Chatterjee} and the invariant-class result from Appendix~\ref{App:Invariance}.

\subsection{Gaussian fluctuations}\label{Sec:GaussianNotation}

In this subsection, we study fluctuations of particle averages around
their self-consistent predictions over a fixed time horizon, starting
from independent initial coordinates.

For each $N\ge2$, consider a coupled map $\bo F^{(N)}$ of the form
\eqref{eq:defF}. The uncoupled maps $f_i$, the interactions $h_{ij}$, and the initial
law may depend on $N$, but we suppress this dependence in the notation
for readability.

Let $X(0)$ have law $\mu_0\in\mc Y_N$. We write
\[
\begin{gathered}
X(t)=\bo F^t(X(0)),\qquad
\mu_t=\bo F_*^t\mu_0,\qquad
\nu_t=\bo{\mc F}^{\,t}\mu_0,\\
\mu_{i,t}=\Pi_i\mu_t,\qquad
\nu_{i,t}=\Pi_i\nu_t,\qquad
p_{i,t}=\frac{d\mu_{i,t}}{d\lambda},\qquad
\rho_{i,t}=\frac{d\nu_{i,t}}{d\lambda}.
\end{gathered}
\]
The self-consistent marginals satisfy $\nu_{i,0}=\mu_{i,0}$ and
\begin{equation}\label{Eq:GFMeanField}
\begin{aligned}
\nu_{i,t+1}&=(G_{i,t})_*\nu_{i,t},
\qquad G_{i,t}:=F_{\nu_t,i},\\
G_{i,t}(x)&=f_i(x)+\frac1N\sum_{j\ne i}
\int_\T h_{ij}(x,y)\,d\nu_{j,t}(y)\pmod1.
\end{aligned}
\end{equation}

\begin{theorem}[Gaussian approximation]\label{Thm:GaussianFluctuations}
Fix $K,\kappa,E$ as in Theorem~\ref{Thm:Main}. Suppose that, for every
$N\ge2$, the map $\bo F^{(N)}$ satisfies
Assumption~\ref{Ass:CondF} with the same bounds $\kappa,K,E$, and that
\begin{equation}\label{Eq:GFUpperDerivative}
\sup_{N\ge2}\max_{1\le i\le N}\|f_i'\|_\infty
\le\Lambda<\infty.
\end{equation}
Suppose also that, for every $N\ge2$, $X(0)$ has product law
$\mu_0=\bigotimes_{i=1}^N\mu_{i,0}\in\mc Y_N(a,M,\gamma)$.
Fix integers $T\ge0$ and $m\ge1$, times
$t_r\in\{0,\ldots,T\}$, and real functions
$\varphi_{i,r}\in C^2(\T)$ satisfying
\[
\sup_{N\ge2}\max_{\substack{1\le i\le N\\1\le r\le m}}
\|\varphi_{i,r}\|_{C^2}\le A.
\]
Define the random vector $Z_N\in\R^m$ by
\[
Z_N^r=\frac1{\sqrt N}\sum_{i=1}^N
\bigl[\varphi_{i,r}(X_i(t_r))
-\nu_{i,t_r}(\varphi_{i,r})\bigr],
\qquad 1\le r\le m,
\]
and put $\Sigma_N:=\operatorname{Cov}(Z_N)$. If
$\mathcal G_N\sim\mathcal N(0,\Sigma_N)$, then, for every
$\Psi\in C_b^3(\R^m)$,
\begin{equation}\label{Eq:GFApproximation}
\left|\E\Psi(Z_N)-\E\Psi(\mathcal G_N)\right|
\le\frac{C_T}{\sqrt N}
\sum_{q=1}^3\|D^q\Psi\|_\infty.
\end{equation}
\end{theorem}

Theorem~\ref{Thm:GaussianFluctuations} describes the fluctuations of
finitely many observables. The following corollary considers the
corresponding fluctuation fields.

\begin{corollary}[Distribution-valued fluctuations]\label{Cor:GaussianField}
Under the assumptions of Theorem~\ref{Thm:GaussianFluctuations}, set
\[
\eta_{N,t}=\frac1{\sqrt N}\sum_{i=1}^N
(\delta_{X_i(t)}-\nu_{i,t}),\qquad 0\le t\le T.
\]
The laws of $(\eta_{N,t})_{t=0}^T$ are tight in
$(H^{-\alpha}(\T))^{T+1}$ for every $\alpha>1/2$.
If
\begin{equation}\label{covariancehpy}
\lim_{N\to\infty}
\operatorname{Cov}\bigl(\eta_{N,s}(\varphi),\eta_{N,t}(\psi)\bigr)
\end{equation}
exists for all real trigonometric polynomials $\varphi,\psi$ and
$s,t\in\{0,\ldots,T\}$, then the fields converge in this space to the
centred Gaussian field with these covariances.
\end{corollary}
 
The  proofs of the above results are given in
Section~\ref{Sec:GaussianFluctuations}.

\section{Proof of Theorem \ref{Thm:Main}}\label{Sec:ProofMain}
The proof of Theorem~\ref{Thm:Main} combines two ingredients:
Proposition~\ref{Prop:LipschitzReg2}, which gives sufficient
conditions for a coupled map to preserve $\mc Y_N(a,M,O(E))$, and
the estimates in Lemma~\ref{Lem:S6TransferEstimates} and
Propositions~\ref{Prop:AveragedMeanField} and~\ref{Prop:OneStepTV}
below, which compare the coupled
dynamics $\bo F_*$ with the self-consistent transfer operator
$\bo{\mc F}$. 

\begin{proposition}\label{Prop:LipschitzReg2}
Let us fix  $\kappa>1$ and $K\ge0$ satisfying
\begin{equation}\label{Eq:MainExpansion}
\left(1+\frac K{4\kappa}\right)
\left(1+\frac{\sqrt K}{2}\right)<\kappa.
\end{equation}
For $E>0$ with $\kappa-2E>1$, let $\bo H=(H_1,\ldots,H_N):\T^N\to\T^N$ be a $C^3$ map satisfying
\begin{equation}\label{Eq:CondHFirstDer}
|\partial_iH_i|\ge\kappa-E,
\qquad
|\partial_jH_i|\le\frac EN
\quad (j\ne i).
\end{equation}
and the higher order conditions 
\begin{equation}\label{Eq:CondHSecondDer}
|\partial_{jk}H_i|\le
\begin{cases}
K+E,&\#\{i,j,k\}=1,\\
E/N,&\#\{i,j,k\}=2,\\
0,&\#\{i,j,k\}=3,
\end{cases}
\end{equation}
and
\begin{equation}\label{Eq:SecondDerivativeJacobianEntries}
|\partial_{jk\ell}H_i|\le
\begin{cases}
K+E,&\#\{i,j,k,\ell\}=1,\\
E/N,&\#\{i,j,k,\ell\}=2,\\
0,&\#\{i,j,k,\ell\}\ge3.
\end{cases}
\end{equation}
Then, if $E$ is sufficiently small (depending on $K$ and $\kappa$),
there exist $a>0$, $M\ge1, \gamma \in\left(0,1\right)$, depending
only on $K,\kappa,E$ (and not on $\bo H$ or $N$), with $a,M$
independent of $E$ and $\gamma=O(E)$,  such that, for every $N\ge2$,
\[
\bo H_*\mc Y_N(a,M,\gamma)
\subseteq\mc Y_N(a,M,\gamma). 
\]
\end{proposition}

The proof of the above proposition is given in Appendix~\ref{App:Invariance}.

The remaining estimates compare the coupled dynamics with its mean-field approximation.

\begin{lemma}\label{Lem:S6TransferEstimates}
Let $T,S:\mathbb T\to\mathbb T$ be $C^3$ local diffeomorphisms
of the same degree, with
\begin{equation}\label{tands}
\inf|T'|,\inf|S'|\ge s>1,
\qquad
\|T''\|_\infty,\|T^{(3)}\|_\infty,
\|S''\|_\infty,\|S^{(3)}\|_\infty\le D.
\end{equation}
Fix lifts $\widetilde T,\widetilde S$ of this degree. Then
\begin{equation}\label{samemap}
d_{\partial}(T_*p,T_*q)
\le
\left(\frac1s+\frac D{s^2}\right)d_{\partial}(p,q),
\end{equation}
where $d_{\partial}$ is as in \eqref{metric}.
Moreover, for every $a,R>0$, there exists $C_0>0$, depending only on
$a,R,s,D$, such that every positive $C^2$ probability density $p$ satisfying
\begin{equation}\label{formulaforp}
\|(\log p)'\|_\infty\le a,\qquad \|p''/p\|_\infty\le R
\end{equation}
gives
\begin{equation}\label{Eq:S6TransferPerturbation}
d_{\partial}(T_*p,S_*p)
\le C_0\|\widetilde T-\widetilde S\|_{C^2}.
\end{equation}
\end{lemma}
The proof of the above result will be given in Appendix \ref{stabbing}.

Given $i\in\{1,\dots,N\},$ for $x\in \mathbb{T}$ and $z\in \mathbb{T}^{N-1}$  write
\[
R_i(x;z):=\frac1N\sum_{j\ne i}h_{ij}(x,z_j),
\qquad
\overline R_{\mu,i}(x)
:=\int R_i(x;z)\,\widehat\mu_i(dz),
\]
where $\widehat\mu_i:=\widehat\Pi_i\mu$. We index $z$ by the
original coordinate labels $j\ne i$. Define
\begin{equation}\label{deltai}
\Delta_i(x;z):=R_i(x;z)-\overline R_{\mu,i}(x)
=F_i(x;z)-F_{\mu,i}(x),
\end{equation}
with the difference taken using the specified lifts.
\begin{proposition}\label{Prop:AveragedMeanField}
For all $a,M>0$, and $\gamma$ small enough, there exists $C>0$ such that, for every $\mu\in\mathcal Y_N(a,M,\gamma)$ and $1\le i\le N$,
\begin{equation}\label{Eq:S6AveragedPerturbation}
\int\|R_i(\cdot;z)-\overline R_{\mu,i}\|_{C^2}
\,\widehat\mu_i(dz)
\le C E N^{-1/2}.
\end{equation}
\end{proposition}

\begin{proof}
Fix $i\in\{1,\ldots,N\}$ and $\mu\in\mathcal Y_N(a,M,\gamma)$. Then
\begin{equation}\label{firstcoolbound}
\int\|\Delta_i(\cdot;z)\|_{C^2}\,\widehat\mu_i(dz)
\le
\sum_{r=0}^3
\int\|\partial_x^r\Delta_i(\cdot;z)\|_{L^1}
\,\widehat\mu_i(dz).
\end{equation}

Let $z,z'$ differ only in coordinate $j\ne i$.
For every $r\in\{0,1,2,3\}$,
\[
\partial_x^r\Delta_i(x;z)-\partial_x^r\Delta_i(x;z')
=
\frac1N\left[
\partial_x^rh_{ij}(x,z_j)
-\partial_x^rh_{ij}(x,z_j')
\right].
\]
By Assumption~\ref{Ass:CondF}, for every fixed $x$ and $j\ne i$,
\[
\operatorname{Osc}_j\bigl[z\mapsto\partial_x^r\Delta_i(x;z)\bigr]
\le\frac{2E}{N}.
\]
This is the coordinate oscillation from Definition~\ref{def:oscillation}.
Apply Proposition~\ref{Thm:ConcMeasure} to the centred function of the
full configuration that depends only on $z$. For sufficiently small $\gamma$,
\begin{align*}
 \int|\partial_x^r\Delta_i(x;z)|\,\widehat\mu_i(dz)
&=\int_0^\infty
\widehat\mu_i\left(|\partial_x^r\Delta_i(x;z)|\ge t\right)\,dt\\
&\le 2\int_0^\infty
\exp\left(-\frac{Nt^2}{8E^2}\right)\,dt\\
&=\sqrt{8\pi}\,\frac E{\sqrt N},  
\end{align*}
which, combined with \eqref{firstcoolbound}, concludes the proof.
\end{proof}

\begin{proposition}[One-step total variation estimate]\label{Prop:OneStepTV}
Under Assumption~\ref{Ass:CondF}, let $a,M>0$ and $0<\gamma<1$.
There is a constant $C$, independent of $N$, such that
\[
\max_i\|\Pi_i\bo F_*\mu-\Pi_i\bo{\mc F}\mu\|_{\mathrm{TV}}
\le C/N,
\]
for all $\mu\in\mc Y_N(a,M,\gamma).$
\end{proposition}
The proof of this Proposition will be carried in appendix \ref{stabbings}.

\begin{proof}[Proof of Theorem \ref{Thm:Main}]
Fix $K$, let $\kappa$ be large enough that \eqref{Eq:MainExpansion} holds, and let $E$ be small enough that $\kappa-2E>1$ and Proposition~\ref{Prop:LipschitzReg2} applies, with $\gamma<1$ (recall $\gamma=O(E)$). Let  $\bo F$ satisfy Assumption~\ref{Ass:CondF} with this choice of $K,\kappa,E$.  Then, $\bo F$ satisfies \eqref{Eq:CondHFirstDer},\eqref{Eq:CondHSecondDer} and \eqref{Eq:SecondDerivativeJacobianEntries} with $\bo H=\bo F.$ By Proposition \ref{Prop:LipschitzReg2}, $\bo F_*\mc Y_N(a,M,\gamma)\subset\mathcal Y_N(a,M,\gamma)$ for every $N\ge2$. Moreover, for every $\mu\in\mc Y_N(a,M,\gamma)$, the product map $\bo F_\mu$ also satisfies \eqref{Eq:CondHFirstDer}, \eqref{Eq:CondHSecondDer} and \eqref{Eq:SecondDerivativeJacobianEntries} with $\bo H=\bo F_\mu$ (its $i$-th component depends only on $x_i$), and the parameters $a,M,\gamma$ given by Proposition~\ref{Prop:LipschitzReg2} depend only on $K,\kappa,E$. Hence $\bo{\mc F}\mu=(\bo F_\mu)_*\mu\in\mc Y_N(a,M,\gamma)$. This proves \eqref{invcone}.

It remains to prove \eqref{distance}. Let $\mu \in \mathcal{M}_1(\mathbb{T}^N)$ have density $\rho \in \mc Y_N(a,M,\gamma).$ For $i\in\{1,\ldots,N\}$ and  $z\in \mathbb{T}^{N-1}$ let $p_{i\mid-i}(\cdot\mid z)$ be the  conditional density of $\mu$ along coordinate $i$, given the remaining coordinates are equal to $z$. Put $\widehat\mu_i=\widehat\Pi_i\mu$. By disintegration,
\[
\Pi_i\bo F_*\mu
=\int F_i(\cdot;z)_*\mu_{i\mid-i}(\cdot\mid z)\,\widehat\mu_i(dz),
\]
\[
\Pi_i\bo{\mc F}\mu
=\int(F_{\mu,i})_*\mu_{i\mid-i}(\cdot\mid z)\,\widehat\mu_i(dz).
\]
Since $\mu\in\mc Y_N(a,M,\gamma)$ and $p_{i\mid-i}(x\mid z)=\frac{\rho(x,z)}{\int_{\T}\rho(s,z)\,ds},$ we have that
\begin{align*}
\left|\frac{\partial_xp_{i\mid-i}(x\mid z)}{p_{i\mid-i}(x\mid z)}\right|
&=\left|\frac{\partial_i\rho(x,z)}{\rho(x,z)}\right|
\le a,\\
\qquad
\left|\frac{\partial_x^2p_{i\mid-i}(x\mid z)}{p_{i\mid-i}(x\mid z)}\right|
&=\left|\partial_{ii}\log\rho(x,z)
+\bigl(\partial_i\log\rho(x,z)\bigr)^2\right|
\le M+a^2.    
\end{align*}
Note that the maps $T=F_i(\cdot;z)$ and $S=F_{\mu,i}$  satisfy \eqref{tands} with $s=\kappa-E>1$ and
$D=K+E$. Thus, by Lemma~\ref{Lem:S6TransferEstimates} and
Proposition~\ref{Prop:AveragedMeanField},
\begin{align*}
d_\partial(\Pi_i\bo F_*\mu,\Pi_i\bo{\mc F}\mu)
&=
d_\partial\left(
\int F_i(\cdot;z)_*\mu_{i\mid-i}(\cdot\mid z)\,\widehat\mu_i(dz),
\int(F_{\mu,i})_*\mu_{i\mid-i}(\cdot\mid z)\,\widehat\mu_i(dz)
\right)\\
&\le
\int d_\partial\bigl(
F_i(\cdot;z)_*\mu_{i\mid-i}(\cdot\mid z),
(F_{\mu,i})_*\mu_{i\mid-i}(\cdot\mid z)
\bigr)\,\widehat\mu_i(dz)\\
&\le
C_0\int
\|F_i(\cdot;z)-F_{\mu,i}\|_{C^2}\,\widehat\mu_i(dz)\\
&=
C_0\int
\|R_i(\cdot;z)-\overline R_{\mu,i}\|_{C^2}\,\widehat\mu_i(dz)\\
&\le \bar C N^{-1/2},
\end{align*}
with $\bar C:= C_0 CE.$
The total variation estimate \eqref{Eq:MainOneStepTV}
follows from Proposition~\ref{Prop:OneStepTV}, after enlarging $\bar C$.

\end{proof}

\section{Proofs of Corollaries \ref{Cor:PropagationofChaos} and \ref{Cor:FixedPointofOperators}}
\label{Sec:ProofMainResBig}

We first record the contraction of the self-consistent transfer operator in $d_N$.

\begin{lemma}\label{Lem:S6MeanFieldLipschitz}
Suppose that $K,\kappa$ satisfy \eqref{Eq:MainExpansion} and $E$ is small enough. Then, there is $L<1$,  such that
\begin{equation}\label{eqest3}
d_N(\bo{\mc F}\mu,\bo{\mc F}\nu)
\le Ld_N(\mu,\nu)
\qquad(\mu,\nu\in\mc Y_N).
\end{equation}
\end{lemma}

\begin{proof}
Fix $i\in\{1,\ldots,N\}$. The formula for the self-consistent evolution in Definition~\ref{meanfieldapprox} gives
\begin{equation}\label{eqest1}
\begin{aligned}
&d_\partial(\Pi_i\bo{\mc F}\mu,\Pi_i\bo{\mc F}\nu)=d_\partial\bigl((F_{\mu,i})_*\Pi_i\mu,
                 (F_{\nu,i})_*\Pi_i\nu\bigr)\\
&\le
d_\partial\bigl((F_{\mu,i})_*\Pi_i\mu,
                 (F_{\mu,i})_*\Pi_i\nu\bigr)
+
d_\partial\bigl((F_{\mu,i})_*\Pi_i\nu,
                 (F_{\nu,i})_*\Pi_i\nu\bigr)    
\end{aligned} 
\end{equation}
For the first term above, note that the maps $T=F_{\mu,i}$ and $S=F_{\nu,i}$  satisfy \eqref{tands} with $s=\kappa-E>1$ and
$D=K+E$. Thus, by \eqref{samemap}
\begin{equation}\label{eqest2}
d_\partial\bigl((F_{\mu,i})_*\Pi_i\mu,
                 (F_{\mu,i})_*\Pi_i\nu\bigr)\le \left(\frac{1}{\kappa-E}+\frac {K+E}{(\kappa-E)^2}\right)d_N(\mu,\nu).
\end{equation}

For the second term, note that \eqref{formulaforp} holds with
$p$ equal to the density of $\Pi_i\nu$ and parameters $a$ and $R=M+a^2$.
By \eqref{Eq:S6TransferPerturbation}, there is $B>0$ such that
\begin{equation}\label{Eq:STOTVFeedback}
\begin{aligned}
d_\partial\bigl((F_{\mu,i})_*\Pi_i\nu,
                 (F_{\nu,i})_*\Pi_i\nu\bigr)
&\le C_0\|F_{\mu,i}-F_{\nu,i}\|_{C^2}\\
&\le BE\max_j\|\Pi_j\mu-\Pi_j\nu\|_{\mathrm{TV}}\\
&\le 2BE\,d_N(\mu,\nu).
\end{aligned}
\end{equation}
Combining this with \eqref{eqest1} and \eqref{eqest2} proves
\eqref{eqest3} with
\[
L=\frac1{\kappa-E}+\frac{K+E}{(\kappa-E)^2}+2BE<1
\]
when $E$ is sufficiently small. Indeed, $\frac1\kappa+\frac K{\kappa^2}<1$:
\eqref{Eq:MainExpansion} gives $\kappa>1+\frac{\sqrt K}2$ and
$\kappa-1>\frac{\sqrt K}2+\frac{K(2+\sqrt K)}{8\kappa}$, whence
$\kappa(\kappa-1)>\frac{\sqrt K}2+\frac K2+\frac{K^{3/2}}8
=K+\frac{\sqrt K}{8}(\sqrt K-2)^2\ge K$.
\end{proof}
Finally, to pass from one-coordinate estimates to estimates for finite collections of coordinates, we will use the following factorisation bound. \begin{lemma}\label{Lem:S6Marginalisation}
Let $\mu\in \mc Y_N(a,M,\gamma)$ with $\gamma\in(0,1)$, and let $S\subset\{1,\ldots,N\}$
have cardinality $k\ge1$. Then, we have
\begin{equation}\label{Eq:S6Factorisation}
\left\|\Pi_S\mu-\bigotimes_{i\in S}\Pi_i\mu
\right\|_{\mathrm{TV}}
\le\frac{5\gamma k(k-1)}{N}.
\end{equation}
\end{lemma}

The proof of this lemma is carried out in Appendix \ref{stabbings}.

\begin{proof}[Proof of Corollary~\ref{Cor:PropagationofChaos}]
Let $\mu_t,\nu_t,\mu_{i,t},\nu_{i,t},p_{i,t},\rho_{i,t}$ be as
introduced in Section~\ref{Sec:GaussianNotation}, with $\mu_0=\mu$.
By \eqref{invcone}, both $\mu_t$ and $\nu_t$ belong to $\mc Y_N$
for every $t\ge0$.
Let $\mathcal L_{i,t}$ and $\mathcal M_{i,t}$ denote the transfer
operators of $F_{\mu_t,i}$ and $F_{\nu_t,i}$, respectively. Define
\[
u_{i,t}:=p_{i,t}-\rho_{i,t},
\qquad
a_t:=\max_{1\le i\le N}\|u_{i,t}\|_1,
\qquad
\varepsilon_{i,t}:=p_{i,t+1}-\mathcal L_{i,t}p_{i,t},
\qquad
\tau_{i,t}:=(\mathcal L_{i,t}-\mathcal M_{i,t})\rho_{i,t}.
\]
Then
\begin{equation}\label{Eq:ChaosErrorRecursion}
u_{i,t+1}
=\mathcal L_{i,t}u_{i,t}+\varepsilon_{i,t}+\tau_{i,t}.
\end{equation}
Taking $L^1$-norms gives
\begin{equation}\label{Eq:ChaosFixedRecursion}
\|u_{i,t+1}\|_1
\le
\|\varepsilon_{i,t}\|_1
+\|\mathcal L_{i,t}u_{i,t}\|_1
+\|\tau_{i,t}\|_1.
\end{equation}

For the first term, observe that $p_{i,t+1}$ is the density of
$\Pi_i\bo F_*\mu_t$, while $\mathcal L_{i,t}p_{i,t}$ is the
density of $\Pi_i\bo{\mc F}\mu_t$. Thus, by the one-step estimate
\eqref{Eq:MainOneStepTV},
\begin{equation}\label{Eq:ChaosFixedOneStep}
\|\varepsilon_{i,t}\|_1
=\|\Pi_i\bo F_*\mu_t-\Pi_i\bo{\mc F}\mu_t\|_{\mathrm{TV}}
\le \frac{C}{N}.
\end{equation}

For the second term, since $\mathcal L_{i,t}$ is an
$L^1$-contraction,
\begin{equation}\label{Eq:ChaosFixedContraction}
\|\mathcal L_{i,t}u_{i,t}\|_1
\le\|u_{i,t}\|_1
\le a_t.
\end{equation}

For the third term, first observe that, since transfer operators
preserve integrals,
\[
\int_\T\tau_{i,t}\,dx
=\int_\T\mathcal L_{i,t}\rho_{i,t}\,dx
-\int_\T\mathcal M_{i,t}\rho_{i,t}\,dx
=0.
\]
Formula \eqref{Eq:STOTVFeedback}, combined with the zero-mean inequality 
\begin{equation}\label{zeromean}
\|v\|_1\le2\|v'\|_1,
\end{equation}
gives
\begin{equation}\label{taui}
\|\tau_{i,t}\|_1\le 2\|\tau_{i,t}'\|_1 \le 2BE\max_{1\le j\le N}
\|\Pi_j\mu_t-\Pi_j\nu_t\|_{\mathrm{TV}}=2BEa_t.    
\end{equation}
Combining \eqref{Eq:ChaosFixedRecursion},
\eqref{Eq:ChaosFixedOneStep}, \eqref{Eq:ChaosFixedContraction}
with the above, we find
\begin{equation}\label{wefind}
a_{t+1}\le(1+2BE)a_t+\frac{C}{N}.
\end{equation}

Since $a_0=0$, we have, for every $T\ge1$,
\begin{equation}\label{Eq:ChaosFixedMarginal}
\max_{0\le t\le T}a_t
\le \frac{C}{N}\sum_{r=0}^{T-1}(1+2BE)^r
\le \frac{\widetilde C_T}{N},
\end{equation}
for some $\widetilde C_T>0$. Finally, let $S\subseteq\{1,\ldots,N\}$ have cardinality $k$.
Combining the above estimate with Lemma~\ref{Lem:S6Marginalisation}
and the triangle inequality, we obtain, for every $0\le t\le T$,
\[
\begin{aligned}
\left\|\Pi_S\mu_t-\bigotimes_{i\in S}\Pi_i\nu_t\right\|_{\mathrm{TV}}
&\le
\left\|\Pi_S\mu_t-\bigotimes_{i\in S}\Pi_i\mu_t\right\|_{\mathrm{TV}}
+\left\|\bigotimes_{i\in S}\Pi_i\mu_t
-\bigotimes_{i\in S}\Pi_i\nu_t\right\|_{\mathrm{TV}}\\
&\le \frac{5\gamma k(k-1)}{N}
+\sum_{i\in S}\|\Pi_i\mu_t-\Pi_i\nu_t\|_{\mathrm{TV}}\\
&\le \frac{5\gamma k(k-1)}{N}+ka_t\\
&\le \frac{5\gamma k(k-1)+k\widetilde C_T}{N}\\
&\le \frac{C_Tk^2}{N}.
\end{aligned}
\] This proves
\eqref{Eq:ChaosFixedTimeTV}.

We now prove \eqref{Eq:ChaosUniformTV}. Put
\[
q:=\frac1{\kappa-E}+\frac{K+E}{(\kappa-E)^2}<1.
\]
For every zero-mean function $v$, the $L^1$-contraction of
$\mathcal L_{i,t}$ and \eqref{samemap} give
\begin{equation}\label{zeromeancontro}
\|\mathcal L_{i,t}v\|_1\le\|v\|_1,
\qquad
\|(\mathcal L_{i,t}v)'\|_1\le q\|v'\|_1.
\end{equation}
For $0\le s<t$, define
\[
\mathcal P_{i;t,s}
:=\mathcal L_{i,t-1}\cdots\mathcal L_{i,s},
\qquad
\mathcal P_{i;t,t}:=\Id.
\]

We claim that
\begin{equation}\label{Eq:ChaosTelescoped}
u_{i,t}
=\sum_{s=0}^{t-1}
\mathcal P_{i;t,s+1}
\bigl(\varepsilon_{i,s}+\tau_{i,s}\bigr).
\end{equation}
Indeed, by \eqref{Eq:ChaosErrorRecursion},
\[
u_{i,s+1}-\mathcal L_{i,s}u_{i,s}
=\varepsilon_{i,s}+\tau_{i,s}.
\]
Applying $\mathcal P_{i;t,s+1}$ to both sides and using
$\mathcal P_{i;t,s+1}\mathcal L_{i,s}=\mathcal P_{i;t,s}$, we get
\[
\mathcal P_{i;t,s+1}u_{i,s+1}
-\mathcal P_{i;t,s}u_{i,s}
=
\mathcal P_{i;t,s+1}
\bigl(\varepsilon_{i,s}+\tau_{i,s}\bigr).
\]
Summing over $s=0,\ldots,t-1$, the terms on the left-hand side
telescope and yield
\[
u_{i,t}-\mathcal P_{i;t,0}u_{i,0}
=\sum_{s=0}^{t-1}
\mathcal P_{i;t,s+1}
\bigl(\varepsilon_{i,s}+\tau_{i,s}\bigr).
\]
Since $u_{i,0}=0$, this proves \eqref{Eq:ChaosTelescoped}.
Taking norms in \eqref{Eq:ChaosTelescoped} gives
\begin{equation}\label{Eq:ChaosTelescopedBound}
\|u_{i,t}\|_1
\le\sum_{s=0}^{t-1}
\left(
\|\mathcal P_{i;t,s+1}\varepsilon_{i,s}\|_1
+\|\mathcal P_{i;t,s+1}\tau_{i,s}\|_1
\right).
\end{equation}
Both $\varepsilon_{i,s}$ and $\tau_{i,s}$ have zero integral.
Combining \eqref{zeromean} with the second inequality in
\eqref{zeromeancontro} gives
$\|\mathcal P_{i;t,s+1}\varepsilon_{i,s}\|_1
\le 2q^{t-1-s}\|\varepsilon_{i,s}'\|_1$,
while the first inequality in \eqref{zeromeancontro} gives
$\|\mathcal P_{i;t,s+1}\varepsilon_{i,s}\|_1
\le\|\varepsilon_{i,s}\|_1$.
Since $\|\varepsilon_{i,s}\|_1\le C/N$ by
\eqref{Eq:ChaosFixedOneStep} and
$\|\varepsilon_{i,s}'\|_1\le C/\sqrt N$ by \eqref{distance},
we obtain
\[
\begin{aligned}
\|\mathcal P_{i;t,s+1}\varepsilon_{i,s}\|_1
&\le
\min\left\{
\|\varepsilon_{i,s}\|_1,\,
2q^{t-1-s}\|\varepsilon_{i,s}'\|_1
\right\}\\
&\le C\min\left\{
\frac1N,\frac{q^{t-1-s}}{\sqrt N}
\right\}.
\end{aligned}
\]
An analogous argument gives
\[
\|\mathcal P_{i;t,s+1}\tau_{i,s}\|_1
\le 2q^{t-1-s}\|\tau_{i,s}'\|_1
\le 2BEq^{t-1-s}a_s,
\]
where the last inequality follows from \eqref{taui}.
Substituting these estimates into \eqref{Eq:ChaosTelescopedBound}
and taking the maximum over $i$ gives
\begin{equation}\label{Eq:ChaosUniformRecursion}
a_t\le
C\sum_{s=0}^{t-1}
\min\left\{\frac1N,\frac{q^{t-1-s}}{\sqrt N}\right\}
+2BE\sum_{s=0}^{t-1}q^{t-1-s}a_s.
\end{equation}
To estimate the first sum, put
\[
R_N:=\left\lceil\frac{\log N}{2|\log q|}\right\rceil.
\]
Then $q^{R_N}\le N^{-1/2}$. Changing variables to $r=t-1-s$
and splitting the sum at $R_N$, we obtain
\begin{align*}
\sum_{s=0}^{t-1}
\min\left\{\frac1N,\frac{q^{t-1-s}}{\sqrt N}\right\}
&=\sum_{r=0}^{t-1}
\min\left\{\frac1N,\frac{q^r}{\sqrt N}\right\}\\
&\le \sum_{r=0}^{R_N-1}\frac1N
+\sum_{r=R_N}^{\infty}\frac{q^r}{\sqrt N}\\
&=\frac{R_N}{N}
+\frac{q^{R_N}}{(1-q)\sqrt N}\\
&\le C\frac{\log(2N)}N.
\end{align*}

Now put $A:=\sup_{t\ge0}a_t$. The above computations, combined with \eqref{Eq:ChaosUniformRecursion}, give 
\begin{align*}
a_t &\le  C\frac{\log(2N)}N + 2BE\sum_{s=0}^{t-1}q^{t-1-s}a_s \\
&\le  C\frac{\log(2N)}N+ 2BEA\sum_{r=0}^{\infty}q^r\\
&= C\frac{\log(2N)}N+\frac{2BE}{1-q}A. 
\end{align*}
Taking the supremum over $t\ge0$ and choosing $E$ sufficiently small we get 
\[
A\le
\frac{C}{1-\frac{2BE}{1-q}}
\frac{\log(2N)}N.
\]
Formula \eqref{Eq:ChaosUniformTV} then follows from the same
telescoping argument used to prove \eqref{Eq:ChaosFixedTimeTV}
\end{proof}

\begin{proof}[Proof of Corollary~\ref{Cor:FixedPointofOperators}]
Let $\bar C>0$ be the constant in Theorem~\ref{Thm:Main},
and put $c=2\bar C/(1-\lambda_{\mathrm{fp}})$ and $r_N=cN^{-1/2}$.
For sufficiently large $N$, $r_N<\delta$.
If $\mu\in B_{r_N}^N(\bar\mu_N)$, the stability hypothesis
and Theorem~\ref{Thm:Main} give
\begin{align*}
d_N(\bo F_*\mu,\bar\mu_N)
&\le \bar C N^{-1/2}+d_N(\bo{\mc F}\mu,\bar\mu_N)\\
&\le\frac{1-\lambda_{\mathrm{fp}}}{2}r_N
+\lambda_{\mathrm{fp}} d_N(\mu,\bar\mu_N)\\
&\le\frac{1+\lambda_{\mathrm{fp}}}{2}r_N<r_N.
\end{align*}
Since $\bo F_*\mu\in\mc Y_N$, this proves the assertion.
\end{proof}

\begin{remark}
If $E$ is small enough for Lemma~\ref{Lem:S6MeanFieldLipschitz} to
apply, the stability hypothesis of
Corollary~\ref{Cor:FixedPointofOperators} holds automatically with
$\delta=\infty$ and $\lambda_{\mathrm{fp}}=L$, since
$d_N(\bo{\mc F}\mu,\bar\mu_N)=d_N(\bo{\mc F}\mu,\bo{\mc F}\bar\mu_N)
\le L\,d_N(\mu,\bar\mu_N)$ for every $\mu\in\mc Y_N$. In that case
the condition $N\ge N_0$ is only needed to guarantee the existence of
$\bar\mu_N$.
\end{remark}

\section{Gaussian Fluctuations}\label{Sec:GaussianFluctuations}
Fix $K,\kappa,E$ as in Theorem~\ref{Thm:Main}. For each $N\ge2$,
let $\bo F^{(N)}:\T^N\to\T^N$ satisfy
Assumption~\ref{Ass:CondF} with these bounds and the additional
condition \eqref{Eq:GFUpperDerivative}. Let $X(0)$ have product law
$\mu_0=\bigotimes_{i=1}^N\mu_{i,0}\in\mc Y_N(a,M,\gamma)$, and let $\mu_t,\nu_t,\mu_{i,t},\nu_{i,t}$ be as defined in  section~\ref{Sec:GaussianNotation}. Furthermore, write $U=(U_1,\ldots,U_N):=X(0)$ and
$\Phi_i^t:=(\bo F^t)_i$, so that $X_i(t)=\Phi_i^t(U)$. Given times $t_1,\ldots,t_m\in\{0,\ldots,T\}$ and test functions
$\varphi_{i,r}$ satisfying the assumptions of Theorem \ref{Thm:GaussianFluctuations}, define the
random vector $W_N=(W_N^1,\ldots,W_N^m)$ by
\[
W_N^r(U)
:=\frac1{\sqrt N}\sum_{i=1}^N
\varphi_{i,r}(\Phi_i^{t_r}(U)),
\qquad 1\le r\le m.
\]
By the definition of $Z_N$, we have
\[
Z_N^r
=W_N^r-\frac1{\sqrt N}\sum_{i=1}^N
\nu_{i,t_r}(\varphi_{i,r}).
\]

We first estimate the dependence of $W_N$ on the initial coordinates.
We then use these estimates to approximate $W_N-\E W_N$ by a sum of
independent random vectors.

\begin{lemma}\label{Lem:GFDerivatives}
For $0\le t\le T$ and $j\ne k$,
\begin{equation}\label{Eq:GFFlowDerivatives}
|\partial_j\Phi_i^t|\le C_T(\mathbf1_{i=j}+N^{-1}),\qquad
\sum_i|\partial_{jk}\Phi_i^t|\le C_T/N.
\end{equation}
In particular, for the vector-valued function
\[
W_N^r(U)=\frac1{\sqrt N}\sum_i\varphi_{i,r}(\Phi_i^{t_r}(U)),
\qquad 1\le r\le m,
\]
we have
\begin{equation}\label{Eq:GFObservableDerivatives}
\|\partial_jW_N\|_\infty\le C_TN^{-1/2},\qquad
\|\partial_{jk}W_N\|_\infty\le C_TN^{-3/2}\quad(j\ne k).
\end{equation}
\end{lemma}
\begin{proof}
Let $P_{ij}=1/N$, so that $P^2=P$. For fixed $b,c>0$, the entrywise bound
$|D\bo F|\le bI+cP$ gives
\[
|D\bo F^t|\le(bI+cP)^t
=b^tI+((b+c)^t-b^t)P.
\]
This proves the first estimate in \eqref{Eq:GFFlowDerivatives}. In particular,
writing $J^t_{ij}=\partial_j\Phi_i^t$, we have
\begin{equation}\label{tobeusedlater}
\sup_x\max_j\sum_i|J^t_{ij}(x)|\le C_T,\qquad
\sup_x\max_{j\ne k}\sum_i|J^t_{ij}(x)J^t_{ik}(x)|\le C_T/N.
\end{equation}
The pairwise form of the interactions also gives
\begin{equation}\label{tobeused2}
\sum_i|\partial_{\ell\ell}F_i|\le C,\qquad
\sum_i|\partial_{\ell q}F_i|\le C/N\quad(\ell\ne q).
\end{equation}

Let $K^t_{i,jk}=\partial_{jk}\Phi_i^t$ and
$B_t=\sup_x\max_{j\ne k}\sum_i|K^t_{i,jk}(x)|$. The chain rule gives
\[
K^{t+1}_{i,jk}
=\sum_\ell\partial_\ell F_i(\Phi^t)K^t_{\ell,jk}
 +\sum_{\ell,q}\partial_{\ell q}F_i(\Phi^t)J^t_{\ell j}J^t_{qk}.
\]
Separating the terms $\ell=q$ and $\ell\ne q$ and using \eqref{tobeusedlater} and \eqref{tobeused2} gives the second bound in \eqref{Eq:GFFlowDerivatives}.
Finally,
\begin{align*}
\partial_jW_N^r
&=\frac1{\sqrt N}\sum_i\varphi_{i,r}'(\Phi_i^{t_r})J^{t_r}_{ij},\\
\partial_{jk}W_N^r
&=\frac1{\sqrt N}\sum_i\left[
\varphi_{i,r}''(\Phi_i^{t_r})J^{t_r}_{ij}J^{t_r}_{ik}
+\varphi_{i,r}'(\Phi_i^{t_r})K^{t_r}_{i,jk}\right].
\end{align*}
These identities, combined with \eqref{Eq:GFFlowDerivatives}, imply \eqref{Eq:GFObservableDerivatives}.
\end{proof}

\begin{lemma}[Independent projection]\label{Lem:GFProjection}
Put  $U=(U_1,...,U_N)=X(0)$. Define $g_{j,N}(U_j)=\E[W_N\mid U_j]-\E W_N$. Then
\begin{equation}\label{Eq:GFProjection}
W_N-\E W_N=\sum_jg_{j,N}(U_j)+R_N,
\qquad |g_{j,N}|\le C_TN^{-1/2},\qquad
\E|R_N|^2\le C_T/N.
\end{equation}
The vectors $g_{j,N}(U_j)$ are independent and centred, and $R_N$ is
orthogonal in $L^2$ to every function of a single initial coordinate.
\end{lemma}
\begin{proof}
By the Hoeffding decomposition for functions of independent random
variables \cite[Section~2]{EfronStein1981}, we can write
\[
W_N-\E W_N=\sum_{j=1}^N g_{j,N}(U_j)+R_N,
\]
where 
$$R_N=\sum_{\substack{B\subseteq\{1,\ldots,N\}\\ |B|\ge2}}W_B(U_B),$$  $U_B:=(U_j)_{j\in B}$ and $W_B(U_B)$ denotes the Hoeffding component indexed by $B$.
The vectors
$g_{j,N}(U_j)$ are independent and centred, and $R_N$ is orthogonal
in $L^2$ to every function of a single initial coordinate.

Write $z=(u_i)_{i\ne j}$ and let $W_N(x;z)$ denote $W_N$ evaluated
at the point whose $j$th coordinate is $x$ and whose remaining
coordinates are given by $z$. Since the initial law is a product,
write $\mu_{-j,0}:=\bigotimes_{i\ne j}\mu_{i,0}.$ By \eqref{Eq:GFObservableDerivatives}, we obtain
\[
\begin{aligned}
|g_{j,N}(x)|
&=\left|\E[W_N\mid U_j=x]-\E W_N\right|\\
&=\left|
\int_{\T^{N-1}}W_N(x;z)\,d\mu_{-j,0}(z)
-\int_{\T^{N-1}}\int_\T
W_N(y;z)\,d\mu_{j,0}(y)\,d\mu_{-j,0}(z)
\right|\\
&=\left|
\int_{\T^{N-1}}\int_\T
\bigl[W_N(x;z)-W_N(y;z)\bigr]\,
d\mu_{j,0}(y)\,d\mu_{-j,0}(z)
\right|\\
&\le \|\partial_j W_N\|_\infty
\le C_TN^{-1/2}.
\end{aligned}
\]
Finally, let $U'=(U'_1,\ldots,U'_N)$ be an independent copy of
$U$. For $j\ne k$, write $z=(U_i)_{i\ne j,k}$ and let
$W_N(a,b;z)$ denote $W_N$ evaluated with its $j$th and $k$th
coordinates equal to $a$ and $b$, respectively, and its remaining
coordinates given by $z$. Define
\[
\begin{aligned}
\Delta_j\Delta_k W_N
:={}&W_N(U_j,U_k;z)-W_N(U'_j,U_k;z)\\
&-W_N(U_j,U'_k;z)+W_N(U'_j,U'_k;z).
\end{aligned}
\]
Since $R_N$ contains only Hoeffding components of order at least
two, the higher-order variance inequality
\cite[Theorem~1.8]{BobkovGotzeSambale2019}, applied componentwise
with $d=2$, yields
\[
\begin{aligned}
\E|R_N|^2
\le \frac14\sum_{j<k}\E|\Delta_j\Delta_k W_N|^2 \frac14\sum_{j<k}\|\partial_{jk}W_N\|_\infty^2
\le \frac{C_T}{N},
\end{aligned}
\]
where the last inequality follows from
\eqref{Eq:GFObservableDerivatives}.
\end{proof}

\begin{proof}[Proof of Theorem~\ref{Thm:GaussianFluctuations}]
Put $S_N:=\sum_{j=1}^N g_{j,N}(U_j),$ $C_N:=\operatorname{Cov}(S_N),$ and let $G_{S,N}\sim\mathcal N(0,C_N)$. Define also $
\Sigma_N:=\operatorname{Cov}(Z_N),$ and $\mathcal G_N\sim\mathcal N(0,\Sigma_N).$ By
Lemma~\ref{Lem:GFProjection}, we have
\[
W_N-\E W_N=S_N+R_N,
\qquad
\E|R_N|^2\le\frac{C_T}{N}.
\]
Since $Z_N$ and $W_N$ differ by a deterministic vector, it follows
that
\[
Z_N=S_N+R_N+\E Z_N.
\]
By the triangle inequality,
\[
\begin{aligned}
\bigl|\E\Psi(Z_N)-\E\Psi(\mathcal G_N)\bigr|
\le{}&
\bigl|\E\Psi(Z_N)-\E\Psi(S_N)\bigr|\\
&+\bigl|\E\Psi(S_N)-\E\Psi(G_{S,N})\bigr|\\
&+\bigl|\E\Psi(G_{S,N})-\E\Psi(\mathcal G_N)\bigr|.
\end{aligned}
\]
For the first term, the decomposition of $Z_N$ and the bound on $R_N$ in Lemma~\ref{Lem:GFProjection} give
\[
\begin{aligned}
\bigl|\E\Psi(Z_N)-\E\Psi(S_N)\bigr|
\le \|D\Psi\|_\infty
\bigl(\E|R_N|+|\E Z_N|\bigr)\le \frac{C_T}{\sqrt N}\|D\Psi\|_\infty.
\end{aligned}
\]
For the last inequality, we also used
\[
\begin{aligned}
|\E Z_N|
&\le \frac1{\sqrt N}\sum_{i=1}^N
\left(\sum_{r=1}^m
|(\mu_{i,t_r}-\nu_{i,t_r})(\varphi_{i,r})|^2\right)^{1/2}
\le \frac{C_T}{\sqrt N},
\end{aligned}
\]
which follows from the propagation of chaos estimate in
Corollary~\ref{Cor:PropagationofChaos}, applied with $S=\{i\}$.

For the second term, let $V_{j,N}$ be independent centred Gaussian vectors, also independent of $U$, with $\operatorname{Cov}(V_{j,N})
=\operatorname{Cov}(g_{j,N}(U_j)).$ Thus $\sum_{j=1}^N V_{j,N}$ has the same law as $G_{S,N}$.

For $1\le j\le N$, put
\[
A_{j,N}:=\sum_{k<j}V_{k,N}
+\sum_{k>j}g_{k,N}(U_k).
\]
Replacing one summand at a time gives
\[
\begin{aligned}
\E\Psi(S_N)-\E\Psi(G_{S,N})
=\sum_{j=1}^N\Bigl[
\E\Psi\bigl(A_{j,N}+g_{j,N}(U_j)\bigr)-\E\Psi\bigl(A_{j,N}+V_{j,N}\bigr)
\Bigr].
\end{aligned}
\]
The vectors $g_{j,N}(U_j)$ and $V_{j,N}$ are independent of
$A_{j,N}$, centred, and have the same covariance. Therefore, the
linear and quadratic terms in their Taylor expansions around
$A_{j,N}$ cancel after taking expectations. Taylor's formula  and the bound on $g_{j,N}$ in  Lemma~\ref{Lem:GFProjection} imply
\[
\begin{aligned}
&\left|
\E\Psi\bigl(A_{j,N}+g_{j,N}(U_j)\bigr)
-\E\Psi\bigl(A_{j,N}+V_{j,N}\bigr)
\right|\\
&\qquad\le
\frac{\|D^3\Psi\|_\infty}{6}
\left(\E|g_{j,N}(U_j)|^3+\E|V_{j,N}|^3\right)
\le \frac{C_T}{N^{3/2}}\|D^3\Psi\|_\infty.
\end{aligned}
\]
Summing over $j=1,\ldots,N$, we conclude that
\[
\bigl|\E\Psi(S_N)-\E\Psi(G_{S,N})\bigr|
\le \frac{C_T}{\sqrt N}\|D^3\Psi\|_\infty.
\]

For the third term, the orthogonality of $R_N$ established  in
Lemma~\ref{Lem:GFProjection} gives $\Sigma_N=C_N+\operatorname{Cov}(R_N),$ and  $\operatorname{tr}\operatorname{Cov}(R_N)
=\E|R_N|^2\le \frac{C_T}{N}.$
Since both $G_{S,N}$ and $\mathcal G_N$ are centered, Taylor's formula yields
\[
\begin{aligned}
\bigl|\E\Psi(G_{S,N})-\E\Psi(\mathcal G_N)\bigr|
&\le \frac12\|D^2\Psi\|_\infty
\operatorname{tr}\operatorname{Cov}(R_N)\\
&\le \frac{C_T}{N}\|D^2\Psi\|_\infty.
\end{aligned}
\]

Combining the three estimates proves \eqref{Eq:GFApproximation}.
\end{proof}

\begin{proof}[Proof of Corollary~\ref{Cor:GaussianField}]
For a real function $\varphi$ with $\|\varphi\|_\infty\le1$, the function
$A_N(x)=N^{-1}\sum_i\varphi(x_i)$ has coordinate oscillations at most $2/N$.
Proposition~\ref{Thm:ConcMeasure} and invariance of $\mc Y_N$ imply
\[
\mathbb P\left(\sqrt N|A_N(X(t))-\E A_N(X(t))|>u\right)
\le2e^{-cu^2},\qquad
N\operatorname{Var}(A_N(X(t)))\le C.
\]
Corollary~\ref{Cor:PropagationofChaos}, applied with $S=\{i\}$,
gives
\[
\left|\frac1{\sqrt N}\sum_i(\mu_{i,t}-\nu_{i,t})(\varphi)\right|\le C.
\]
Let $\eta_{N,t}$ as in the statement of the corollary. For $e_n(x)=e{2\pi\mathrm i nx}$, applying the preceding estimates to
$\operatorname{Re}e_n$ and $\operatorname{Im}e_n$ gives    
\[
\begin{aligned}
\E|\eta_{N,t}(e_n)|^2
&=\E|\eta_{N,t}(\operatorname{Re}e_n)|^2
+\E|\eta_{N,t}(\operatorname{Im}e_n)|^2\\
&=\sum_{\varphi\in\{\operatorname{Re}e_n,\operatorname{Im}e_n\}}
\left[
N\operatorname{Var}(A_N(X(t)))
+\left|\frac1{\sqrt N}\sum_{i=1}^N
(\mu_{i,t}-\nu_{i,t})(\varphi)\right|^2
\right]\\
&\le C_T.
\end{aligned}
\]
Using the Fourier definition of the Sobolev norm, for every $\beta>1/2$, we have
\[
\begin{aligned}
\sup_{N\ge2}\max_{0\le t\le T}
\E\|\eta_{N,t}\|_{H^{-\beta}}^2
&=\sup_{N\ge2}\max_{0\le t\le T}
\E\left[
\sum_{n\in\mathbb Z}
(1+4\pi^2n^2)^{-\beta}
|\eta_{N,t}(e_n)|^2
\right]\\
&=\sup_{N\ge2}\max_{0\le t\le T}
\sum_{n\in\mathbb Z}
(1+4\pi^2n^2)^{-\beta}
\E|\eta_{N,t}(e_n)|^2\\
&\le C_T\sum_{n\in\mathbb Z}
(1+4\pi^2n^2)^{-\beta}
<\infty.
\end{aligned}
\]
Choose $1/2<\beta<\alpha$. The compact embedding
$H^{-\beta}(\T)\hookrightarrow H^{-\alpha}(\T)$ and Markov's inequality
imply joint tightness at the finitely many times $0,\ldots,T$.
if furthermore \eqref{covariancehpy} holds, Theorem~\ref{Thm:GaussianFluctuations}
identifies the joint limit of every finite collection of real Fourier
coordinates. These coordinates generate the Borel sigma-algebra of the
separable Hilbert space $(H^{-\alpha})^{T+1}$, so every subsequential limit
has the same law. This proves convergence of the whole sequence. The limit
is Gaussian because every continuous linear functional is the almost-sure
limit of its Fourier truncations, which are centred Gaussian variables.
\end{proof}

\appendix

\section{Concentration inequalities for quasi-product  measures }\label{Sec:Quasi-Product-Chatterjee}
Recall that, for two Borel probability measures $\nu$ and $\eta$
on $\T$, with densities $p,q$,
\begin{equation}\label{tvdensities}
\|\nu-\eta\|_{\mathrm{TV}}
=
\int_\T|p-q|\,dx
=
2\left(1-\int_\T\min\{p,q\}\,dx\right).
\end{equation}
The following theorem is \cite[Theorem~4.3]{chatterjee}.
\begin{theorem}\label{chatter}
Let $\mu$ be a Borel probability measure on $\T^N$.
Given $x,y\in\T^N$ and  $i\in\{1,\ldots,N\}$, let $\mu_{i\mid-i}(\cdot\mid\widehat x_i)$ and
$\mu_{i\mid-i}(\cdot\mid\widehat y_i)$ denote the corresponding disintegrations along the $i$-th coordinate.
Suppose that $A=(a_{ij})_{i,j=1}^N$ is a matrix with
nonnegative entries and zero diagonal terms satisfying
\begin{equation}\label{chatterjeeecond}
\frac12\left\|
    \mu_{i\mid-i}(\cdot\mid\widehat x_i)
    -\mu_{i\mid-i}(\cdot\mid\widehat y_i)
\right\|_{\mathrm{TV}}
\le
\sum_{j=1}^N a_{ij}\mathbf{1}_{\{x_j\ne y_j\}}
\end{equation}
for every $i$ and every $x,y\in\T^N$.
Assume that $\|A\|_{2\to2}<1$, where $\|\cdot\|_{2\to2}$
denotes the Euclidean operator norm.
Then, for every $\alpha>0$ and every nonconstant
$g\in\mc O_\alpha(\T^N)$,
\[
\mu\left(
    \left|g-\mb E_\mu[g]\right|\ge\epsilon
\right)
\le
2\exp\left[
    -\frac{(1-\|A\|_{2\to2})\epsilon^2}
    {\displaystyle\sum_{i=1}^N
    \left(\operatorname{Osc}_i g\right)^2}
\right]
\qquad\text{for every }\epsilon\ge0.
\]
\end{theorem}
\begin{proof}[Proof of Proposition~\ref{Thm:ConcMeasure}]
Let $\mu\in\mc Y_N(a,M,\gamma)$, with $\gamma\in(0,1)$,
and let $\rho$ be its density. By the proof of \cite[Proposition~8.8, pp.~143--144]{Georgii2011},
for any two positive probability densities $p,q$ on $\T$,
\begin{equation}\label{osctv}
\frac12\|p-q\|_{\mathrm{TV}}
\le\frac14\operatorname{osc}\log\frac pq.
\end{equation}

Fix $i\ne j$ and suppose that $x,y\in\T^N$ differ only in
coordinate $j$. Let $p,q$ be the densities of
$\mu_{i\mid-i}(\cdot\mid\widehat x_i),\mu_{i\mid-i}(\cdot\mid\widehat y_i)$, respectively.
They are given by
\[
p(t)=\frac{\rho(t,\hat x_i)}
{\int_\T\rho(s,\hat x_i)\,ds},
\qquad
q(t)=\frac{\rho(t,\hat y_i)}
{\int_\T\rho(s,\hat y_i)\,ds},
\]
where $t$ occupies the $i$-th coordinate. Since
$\mu\in\mc Y_N(a,M,\gamma)$, we obtain
\begin{align*}
\operatorname{osc}\log\frac pq
&=
\operatorname{osc}_{t\in\T}
\bigl[\log\rho(t,\hat x_i)-\log\rho(t,\hat y_i)\bigr]\\
&\le
\frac12\|\partial_{ij}\log\rho\|_\infty\,d_\T(x_j,y_j)\\
&\le\frac{\gamma}{4N}.
\end{align*}
By \eqref{osctv}, formula \eqref{chatterjeeecond} holds, for such $x,y$, with $A=(a_{ij})_{i,j=1}^N$ where $a_{ij}=0$ when $i=j$ and  $a_{i,j}=\frac{\gamma}{16N}$ for $i\neq j$. For general $x,y\in\T^N$, \eqref{chatterjeeecond} follows from the triangle inequality, changing one coordinate at a time.  Since $\|A\|_{2\to2}
=\frac{\gamma}{16}\frac{N-1}{N}<1,$ for $\gamma<1,$ Proposition \ref{Thm:ConcMeasure} follows from  Theorem~\ref{chatter}.
\end{proof}

\section{An obstruction to a uniform \texorpdfstring{$N^{-1}$}{1/N} rate}
\label{Sec:RateObstruction}
Fix $K>0$ and choose $\kappa$ and $E>0$ satisfying the hypotheses of
Corollary~\ref{Cor:PropagationofChaos}. Choose an odd integer $q\ge5$
sufficiently large, and then choose $\alpha,\varepsilon>0$ sufficiently
small, so that the maps $\bo F^{(N_\ell)}$ defined in
\eqref{Eq:ObstructionArray} satisfy assumption \ref{Ass:CondF} with the same bounds
$\kappa$, $K$, and $E$, for every $\ell\ge1$. The aim of this section is to prove this family of maps satisfies \eqref{Eq:RateObstruction}.

First of all let's argue that the defined coupled maps can be made to satisfy Assumption~\ref{Ass:CondF} provided that the parameters are chosen suitably.  
Recall that $m_\ell=q^\ell$, so that
$B_\ell(y)=y+\frac{\alpha}{2\pi m_\ell^3}\sin(2\pi m_\ell y)$ satisfies
\[
\|B_\ell'-1\|_\infty\le\frac{\alpha}{m_\ell^2}\le\alpha,\qquad
\|B_\ell''\|_\infty=\frac{2\pi\alpha}{m_\ell}\le2\pi\alpha,\qquad
\|B_\ell'''\|_\infty=4\pi^2\alpha .
\]
Differentiating $B_\ell\circ f_1=q\,\cdot$, where $f_1=A_\ell\circ(q\,\cdot)$, gives
\[
f_1'=\frac q{B_\ell'\circ f_1},\qquad
f_1''=-q^2\frac{B_\ell''\circ f_1}{(B_\ell'\circ f_1)^3},\qquad
f_1'''=q^3\left[\frac{3(B_\ell''\circ f_1)^2}{(B_\ell'\circ f_1)^5}
-\frac{B_\ell'''\circ f_1}{(B_\ell'\circ f_1)^4}\right].
\]
Hence, if $\alpha\le1/2$, there is an absolute constant $C_*$ such that,
for every $\ell\ge1$,
\[
\inf f_1'\ge\frac q{1+\alpha},\qquad
\|f_1''\|_\infty\le C_*q^2\alpha,\qquad
\|f_1'''\|_\infty\le C_*q^3\alpha .
\]
The other uncoupled maps are $f_j(x)=qx$, $j\ge2$. The only nonzero
interactions are
$h_{1j}(x,y)=\varepsilon[\sin(2\pi y)-q^{-1}\sin(2\pi qy)]$,
$2\le j\le N_\ell$, and
$\|h_{1j}\|_{C^3}\le(2\pi)^3(1+q^2)\varepsilon$. Therefore
Assumption~\ref{Ass:CondF} holds with the bounds $\kappa$, $K$, $E$,
for every $\ell\ge1$, provided
\[
\frac q{1+\alpha}>\kappa,\qquad C_*q^3\alpha\le K,\qquad
(2\pi)^3(1+q^2)\varepsilon\le E .
\]
These conditions are met by choosing first $q$, then $\alpha$, and then
$\varepsilon$.

Fix $\ell\ge1$ and write
\[
m=q^\ell,\qquad N=m^2,\qquad b=\frac{\alpha}{N},
\qquad T=\ell+1.
\]

Let $Y_t=X_1(t)$ denote the position of the hub associated  the coupled map $F^{(N_\ell)}$ as defined in \eqref{Eq:ObstructionArray}, and put $U=(U_2,\ldots,U_N),$ with $U_j=X_j(0).$ Then $X_j(t)=q^tU_j$ for $j\in \{2,\dots,N_l\}$ and the hub  satisfies
$$Y_{t+1}=A_\ell(qY_t)+\xi_t(U)\pmod 1,$$
\[
\eta_t(U)=\frac{\varepsilon}{N}\sum_{j=2}^N
\sin(2\pi q^tU_j),
\qquad
\xi_t(U)=\eta_t(U)-q^{-1}\eta_{t+1}(U).
\]
The hub dynamics for the STO instead reads 
$$\bar Y_{t+1} = A_\ell(q\bar Y_t),$$
while each leaf follows the same map $x\mapsto qx\pmod 1$ evolution as for $F^{(N_\ell)}$.

Let $\bar p_t$ and $\rho_t$ denote the hub densities at time $t$ under
the coupled and self-consistent evolutions, respectively, given uniformly distributed initial conditions  on $\mathbb{T}^{N_\ell}$. Furthermore, let $p_t^U$ denote the corresponding conditional density of $Y_t$, so that
\begin{equation}\label{average}
\bar p_t=\mathbb E_U[p_t^U],
\end{equation}
where, in the above $\mathbb E$ denotes the expectation with respect to Lebesgue. Then, we have
\begin{equation}\label{recursiondensities}
\begin{aligned}
p_0^U=1,&\qquad p_{t+1}^U=R_{\xi_t(U)}QPp_t^U\\
\rho_0=1,&\qquad \rho_{t+1}=QP\rho_t,
\end{aligned}
\end{equation}

where $P$ and $Q$ denote the transfer operators of $y\mapsto qy$ and
$A_\ell$, respectively, and, for $\xi\in\R$,  $R_\xi$ is the transfer
operator of $y\mapsto y+\xi$ modulo one. In particular, one can compute, for $v\in L^1$, the formulas
\[
(Pv)(y)=\frac1q\sum_{j=0}^{q-1}
v\left(\frac{y+j}{q}\right),\qquad
(Qv)(y)=B_\ell'(y)v(B_\ell(y)),\qquad
(R_\xi v)(y)=v(y-\xi).
\]

Put $c_m(y)=\cos(2\pi my)$. Since 
\[
P1=1,\qquad Q1=B_\ell'=1+bc_m,
\]
then, by \eqref{recursiondensities}, the self-consistent density has the exact representation
\begin{equation}\label{exactrep}
\rho_t
=1+b\sum_{r=0}^{t-1}(QP)^r c_m.
\end{equation}
For the finite-dimensional evolution,  \eqref{recursiondensities} gives
\begin{equation}\label{Eq:ObstructionConditionalDecomposition}
p_{t+1}^U-1
=R_{\xi_t(U)}
\left[
bc_m+P(p_t^U-1)
+(Q-\Id)P(p_t^U-1)
\right].
\end{equation}

Set $z_0^U=0$ and define $z_t^U$ recursively by
\begin{equation}\label{Eq:ObstructionApproximation}
z_{t+1}^U
=R_{\xi_t(U)}\bigl(bc_m+Pz_t^U\bigr).
\end{equation}
The following proposition compares $z_t^U$ and $z_t^{(0)}$ with
$p_t^U$ and $\rho_t$, respectively.

\begin{proposition}\label{decisiveestimate}
For every $0\le t\le T$, we can write
\begin{equation}\label{Eq:ObstructionDensityDecomposition}
p_t^U=1+z_t^U+e_t^U,
\qquad
\rho_t=1+z_t^{(0)}+e_t^{(0)},
\end{equation}
where, uniformly in the initial leaf positions $U$,
\begin{equation}\label{Eq:ObstructionErrorBounds}
\|e_t^U\|_1\le Cb^2t^2,
\qquad
\|e_t^{(0)}\|_1\le Cb^2t^2.
\end{equation}
Moreover, for $1\le t\le T$,
\begin{equation}\label{Eq:ObstructionExplicitZ}
z_t^U(y)
=b\sum_{r=0}^{t-1}
\cos\left(
2\pi\frac{m}{q^r}y
-2\pi m\sum_{j=0}^{r}
q^{-j}\xi_{t-1-r+j}(U)
\right)
\end{equation}
and
\begin{equation}\label{Eq:ObstructionExplicitZ0}
z_t^{(0)}(y)
=b\sum_{r=0}^{t-1}
\cos\left(2\pi\frac{m}{q^r}y\right).
\end{equation}
\end{proposition}

The proof of Proposition~\ref{decisiveestimate} is deferred at the end of this Section. 

We now prove \eqref{Eq:RateObstruction}. By \eqref{average},
averaging the first decomposition in
\eqref{Eq:ObstructionDensityDecomposition} over $U$ and subtracting
the second gives
\[
\bar p_T-\rho_T
=\mathbb E_U[z_T^U]-z_T^{(0)}
+\mathbb E_U[e_T^U]-e_T^{(0)}.
\]
By \eqref{Eq:ObstructionErrorBounds}, it follows that
\begin{equation}\label{Eq:ObstructionReduction}
\|\bar p_T-\rho_T\|_1
\ge
\|\mathbb E_U[z_T^U]-z_T^{(0)}\|_1-2Cb^2T^2.
\end{equation}
By \eqref{Eq:ObstructionExplicitZ} and
\eqref{Eq:ObstructionExplicitZ0}, we have
\begin{align*}
\mathbb E_U[z_T^U(y)]-z_T^{(0)}(y)
&=b\sum_{r=0}^{\ell}\Bigl[
\mathbb E_U\!\left(
\cos(2\pi q^{\ell-r}y-\Theta_r)
\right)
-\cos(2\pi q^{\ell-r}y)
\Bigr]\\
&=b\sum_{r=0}^{\ell}\Bigl[
\cos(2\pi q^{\ell-r}y)
\bigl(\mathbb E_U[\cos\Theta_r]-1\bigr)
+\sin(2\pi q^{\ell-r}y)\mathbb E_U[\sin\Theta_r]
\Bigr].
\end{align*}
where
\[
\Theta_r(U)
:=2\pi m\sum_{j=0}^{r}q^{-j}
\xi_{T-1-r+j}(U).
\]
Note that $\Theta_r$ is odd in $U$. Indeed, each $\eta_t$ is odd in
$U$, and hence so is $\xi_t=\eta_t-q^{-1}\eta_{t+1}$. It follows that $\mathbb E_U[\sin\Theta_r]=0.$ We then have

\begin{equation}\label{zedtu}
\mathbb E_U[z_T^U]-z_T^{(0)}=-bH_\ell,
\qquad
H_\ell(y):=\sum_{r=0}^{\ell}d_r\cos(2\pi q^{\ell-r}y),\qquad d_r:=1-\mathbb E_U[\cos\Theta_r]
\end{equation}

We now prove that $d_r$ is bounded below uniformly in $N$ and $r$.
Fix $0\le r\le\ell$ and put $a=T-1-r$. Since
$\xi_t=\eta_t-q^{-1}\eta_{t+1}$, the sum defining $\Theta_r$
telescopes:
\[
\sum_{j=0}^r q^{-j}\xi_{a+j}
=\eta_a-q^{-r-1}\eta_T.
\]
Define
\[
D_r(u):=\sin(2\pi u)-q^{-r-1}\sin(2\pi q^{r+1}u),
\qquad
\zeta:=\frac{2\pi\varepsilon}{m}.
\]
Using $N=m^2$ and $a+r+1=T$, we obtain
\[
\Theta_r(U)=\zeta\sum_{j=2}^N D_r(q^aU_j).
\]
The variables $q^aU_j$ modulo one are independent and uniformly
distributed on $\T$. Moreover, $D_r$ is odd, so
$\int_\T\sin(\zeta D_r(u))\,du=0$. Consequently,
\[
\mathbb E_U[e^{i\Theta_r}]
=\left(\int_\T e^{i\zeta D_r(u)}\,du\right)^{N-1}
=J_r^{N-1},
\qquad
J_r:=\int_\T\cos(\zeta D_r(u))\,du.
\]

We have $\|D_r\|_\infty\le1+q^{-1}$, and orthogonality of the
two sine functions gives
\[
\int_\T D_r(u)^2\,du
=\frac12(1+q^{-2r-2})\ge\frac12.
\]
Choose $\varepsilon>0$ sufficiently small that
$2\pi\varepsilon(1+q^{-1})\le1$.
Then $|\zeta D_r(u)|\le1$ for every $u$, $r$, and $\ell$.
Using $\cos v\le1-v^2/4$ for $|v|\le1$, we obtain
\[
0<J_r
\le1-\frac{\zeta^2}{4}\int_\T D_r(u)^2\,du
\le1-\frac{\zeta^2}{8}.
\]
Since $N-1\ge N/2$ and $\zeta^2=4\pi^2\varepsilon^2/N$,
\[
\mathbb E_U[\cos\Theta_r]
=J_r^{N-1}
\le \exp\left(-\frac{(N-1)\zeta^2}{8}\right)
\le e^{-\pi^2\varepsilon^2/4}.
\]
Thus
\[
d_r=1-\mathbb E_U[\cos\Theta_r]
\ge\delta:=1-e^{-\pi^2\varepsilon^2/4}>0,
\qquad 0\le r\le\ell,
\]
where $\delta$ is independent of $N$ and $r$. 

In order to prove \eqref{Eq:RateObstruction}, we need a lower bound on $\|H_\ell\|_1.$
To that goal, note that, by H\"older's inequality, we have 
\begin{equation}\label{holderito}
\|H_\ell\|_1
\ge\frac{\|H_\ell\|_2^3}{\|H_\ell\|_4^2}
\ge\frac{\|H_\ell\|_2}{\sqrt3}
\ge\delta\sqrt{\frac{T}{6}}.
\end{equation}

By orthogonality and the bound $d_r\ge\delta$, we have
\begin{equation}\label{elltuito}
\|H_\ell\|_2^2
=\frac12\sum_{r=0}^{\ell}d_r^2
\ge\frac{\delta^2T}{2}.
\end{equation}
To obtain a lower bound on $\|H_\ell\|_1$, we then only need to estimate its fourth moment. Expanding the fourth power gives
\[
\|H_\ell\|_4^4
=\sum_{r,s,u,v=0}^{\ell}d_rd_sd_ud_v\,I_{r,s,u,v},
\]
where
\[
I_{r,s,u,v}
:=\int_{\T}
\cos(2\pi q^{\ell-r}y)
\cos(2\pi q^{\ell-s}y)
\cos(2\pi q^{\ell-u}y)
\cos(2\pi q^{\ell-v}y)\,dy.
\]
We claim that
\[
I_{r,s,u,v}
=
\begin{cases}
\displaystyle\frac38,
&\text{if all four indices are equal},\\[4pt]
\displaystyle\frac14,
&\text{if two distinct indices each occur twice},\\[4pt]
0,
&\text{otherwise}.
\end{cases}
\]
Indeed, writing each cosine as a sum of two complex exponentials,
we see that a term in the product contributes to the integral only
if its frequencies satisfy
\[
\pm q^{\ell-r}\pm q^{\ell-s}
\pm q^{\ell-u}\pm q^{\ell-v}=0.
\]
Since $q\ge5$, the largest power appearing in such a sum must be
cancelled by an equal power with the opposite sign. The remaining
two powers must then cancel as well. Thus $I_{r,s,u,v}$ vanishes
unless the indices occur in equal pairs. Consequently,
\[
\|H_\ell\|_4^4
=\frac38\sum_{r=0}^{\ell}d_r^4
+\frac32\sum_{0\le r<s\le\ell}d_r^2d_s^2
\le 3\|H_\ell\|_2^4.
\]
Combining this bound with \eqref{holderito}, \eqref{elltuito}, \eqref{zedtu} and \eqref{Eq:ObstructionReduction}, we conclude that
\[
\|\bar p_T-\rho_T\|_1
\ge b\sqrt T
\left(\frac{\delta}{\sqrt6}-2CbT^{3/2}\right).
\]
Since $b=\alpha/N$ and $T=1+\log N/(2\log q)$, the expression in
parentheses tends to $\delta/\sqrt6>0$ as $\ell\to\infty$.
This proves \eqref{Eq:RateObstruction}.

\begin{proof}[Proof of Proposition~\ref{decisiveestimate}]
We first prove \eqref{Eq:ObstructionExplicitZ} by induction on $t$.
For $t=1$, since $z_0^U=0$, we have
\[
z_1^U(y)
=R_{\xi_0(U)}(bc_m)(y)
=b\cos\bigl(2\pi my-2\pi m\xi_0(U)\bigr),
\]
which agrees with \eqref{Eq:ObstructionExplicitZ}.

Suppose that \eqref{Eq:ObstructionExplicitZ} holds for some
$1\le t<T$. For brevity, write $\xi_s=\xi_s(U)$. Since
$m=q^\ell$ and $t<T=\ell+1$, each frequency $m/q^r$ appearing
in $z_t^U$ is divisible by $q$. Thus, applying $P$ to
\eqref{Eq:ObstructionExplicitZ} gives
\[
(Pz_t^U)(y)
=b\sum_{r=0}^{t-1}
\cos\left(
2\pi\frac{m}{q^{r+1}}y
-2\pi m\sum_{j=0}^{r}q^{-j}\xi_{t-1-r+j}
\right).
\]
Using the recursion for $z_{t+1}^U$ and applying $R_{\xi_t}$,
we obtain
\[
\begin{aligned}
z_{t+1}^U(y)
&=b\cos\bigl(2\pi m(y-\xi_t)\bigr) +b\sum_{r=0}^{t-1}
\cos\left(
2\pi\frac{m}{q^{r+1}}(y-\xi_t)
-2\pi m\sum_{j=0}^{r}q^{-j}\xi_{t-1-r+j}
\right)\\
&=b\cos(2\pi my-2\pi m\xi_t)
+b\sum_{r=0}^{t-1}
\cos\left(
2\pi\frac{m}{q^{r+1}}y
-2\pi m\sum_{j=0}^{r+1}q^{-j}\xi_{t-1-r+j}
\right)\\
&=b\sum_{r=0}^{t}
\cos\left(
2\pi\frac{m}{q^r}y
-2\pi m\sum_{j=0}^{r}q^{-j}\xi_{t-r+j}
\right).
\end{aligned}
\]
This proves \eqref{Eq:ObstructionExplicitZ} at time $t+1$ and
completes the induction. Formula \eqref{Eq:ObstructionExplicitZ0}
follows in the same way, taking $\xi_t=0$ for every $t$.

To prove \eqref{Eq:ObstructionDensityDecomposition} and
\eqref{Eq:ObstructionErrorBounds}, fix $U$ and put $u_t:=p_t^U-1$ and $e_t^U:=u_t-z_t^U.$ Subtracting
\eqref{Eq:ObstructionApproximation} from
\eqref{Eq:ObstructionConditionalDecomposition} gives
\[
e_{t+1}^U
=R_{\xi_t}\left[
Pe_t^U+(Q-\Id)P(u_t)
\right] =R_{\xi_t}QP e_t^U
+R_{\xi_t}(Q-\Id)Pz_t^U.
\]
Since $P$, $Q$, and $R_{\xi_t}$ are $L^1$-contractions,
\begin{equation}\label{Eq:ObstructionErrorRecursion}
\|e_{t+1}^U\|_1
\le \|e_t^U\|_1
+\|(Q-\Id)Pz_t^U\|_1.
\end{equation}

Since $Q$ is the transfer operator associated to $A_l=B_l^{-1},$ we have that
\begin{align*}
\|(Q-\Id)Pz_t^U\|_1
&\le
\|(B_\ell'-1)(Pz_t^U)\circ B_\ell\|_1
+\|(Pz_t^U)\circ B_\ell-Pz_t^U\|_1\\
&\le
b\|Pz_t^U\|_\infty
+\frac{b}{2\pi m}\|(Pz_t^U)'\|_\infty\\
&\le
b^2t+\frac{b^2}{q-1}
\le Cb^2(t+1),
\end{align*}
where between the first and the second line, we used the inequalities $\|B_\ell'-1\|_\infty\le b$ and $\|B_\ell-\Id\|_\infty\le\frac{b}{2\pi m},$ which follow directly from the definition of $B_l$, and between the second and the third line, we used the fact that, by \eqref{Eq:ObstructionExplicitZ}, $\|Pz_t^U\|_\infty\le bt$ and 
\begin{align*}
(Pz_t^U)'(y)
=-2\pi bm\sum_{r=0}^{t-1}q^{-(r+1)}
\sin\left(
2\pi\frac{m}{q^{r+1}}y
-2\pi m\sum_{j=0}^{r}q^{-j}\xi_{t-1-r+j}
\right),
\end{align*}
so that $\|(Pz_t^U)'\|_\infty
\le 2\pi bm\sum_{r=0}^{t-1}q^{-(r+1)}
\le \frac{2\pi bm}{q-1}.$

It follows from \eqref{Eq:ObstructionErrorRecursion} that
\[
\|e_{t+1}^U\|_1
\le \|e_t^U\|_1+Cb^2(t+1).
\]
Since $e_0^U=0$, summing the preceding inequality over $0\le s<t$  proves the first bound in \eqref{Eq:ObstructionErrorBounds}.
The corresponding decomposition in
\eqref{Eq:ObstructionDensityDecomposition} follows from the
definition of $e_t^U$. The remaining claims follow by the same
argument, taking $\xi_t=0$ for every $t$.
\end{proof}

\section{Proof of invariance}\label{App:Invariance}
\subsection{Preliminary geometric estimates}
Let us record some preliminary estimates, that will be used later. First observe that by \eqref{Eq:CondHSecondDer}
\begin{equation}\label{sumjacobianetries}
\sum_{i,j}|[\partial_kDH]_{ij}|\le K+4E.
\end{equation}  
and by \eqref{Eq:SecondDerivativeJacobianEntries}
\begin{equation}\label{sumsecondjacobianentries}
\sum_{r,s}|[\partial_{jk}DH]_{rs}|\le K+4E.
\end{equation}
We then claim that
\begin{equation}\label{jacobitwo}
\|(DH)^{-1}\|_{p\to p}\le\frac1{\kappa-2E},
\qquad p\in\{1,\infty\},
\end{equation}
and
\begin{equation}\label{Eq:InverseJacobianEntries}
|[(DH)^{-1}]_{ij}|
\le
\begin{cases}
\dfrac1{\kappa-2E}, & i=j,\\[6pt]
\dfrac{E}{(\kappa-E)(\kappa-2E)N}, & i\ne j.
\end{cases}
\end{equation}
To see this, observe that we can write $DH=A+R$, with
$A=\operatorname{diag}(\partial_1H_1,\ldots,\partial_NH_N)$, so that 
\begin{align*}
\bigl[DH\bigr]^{-1}
=\bigl[I+(A^{-1}R)\bigr]^{-1}
A^{-1}
\end{align*}
The diagonal entries of $A^{-1}R$ vanish, while its
off-diagonal entries satisfy 
\begin{equation}\label{arentries}
\bigl|(A^{-1}R)_{rj}\bigr|
\le\frac{E}{(\kappa-E)N}.
\end{equation}
In particular, this implies that 
\begin{equation}\label{diagonalentriesar-1}
\|A^{-1}R\|_{p\to p}\le E/(\kappa-E)<1, \qquad p\in \{1,\infty\}
\end{equation}
The Neumann series identity
\[
[I+A^{-1}R]^{-1}
=\sum_{n\ge0}(-1)^n(A^{-1}R)^n,
\]
together with \eqref{diagonalentriesar-1} and
\eqref{Eq:CondHFirstDer}, gives \eqref{jacobitwo}.
Similarly, using \eqref{arentries} and
\eqref{Eq:CondHFirstDer} gives
\eqref{Eq:InverseJacobianEntries}.

Let 
\begin{equation}\label{gi}
g=\log|\det DH|.
\end{equation}
The following holds.
\begin{lemma}\label{logdetdH}
For all $j =1\dots, N$ we have 
\begin{equation}\label{Eq:GeometryJacobianRational}
\|\partial_jg\|_\infty\le\frac{K+4E}{\kappa-2E},
\end{equation}
and
\begin{equation}\label{Eq:GeometryJacobianSecondRational}
\|\partial_{jj}g\|_\infty
\le\frac{K+4E}{\kappa-2E}
+\frac{(K+E)^2+3E^2}{(\kappa-2E)^2}.
\end{equation}
Furthermore, if  $j\neq k$
\begin{equation}\label{zhengbian}
\|\partial_{jk}g\|_\infty
\le
\frac{E}{N}
\left[
\frac{4}{\kappa-2E}
+
\frac{2K+5E}{(\kappa-2E)^2}
\left(1+\frac{K+E}{\kappa-E}\right)
\right].
\end{equation}
\end{lemma}
\begin{proof}
By Jacobi's formula
\begin{equation}\label{jacobione}
\partial_jg
=\operatorname{tr}\left((DH)^{-1}\partial_jDH\right)
=\sum_{r,s}[(DH)^{-1}]_{r,s}(\partial_jDH)_{s,r}.
\end{equation}

Formula \eqref{Eq:GeometryJacobianRational} follows then combining the above with \eqref{jacobitwo} and \eqref{sumjacobianetries}.

We now prove \eqref{Eq:GeometryJacobianSecondRational}.
Differentiating $DH(DH)^{-1}=I$ gives
\[
\partial_k(DH)^{-1}
=-(DH)^{-1}(\partial_kDH)(DH)^{-1}.
\]
Differentiating \eqref{jacobione} and plugging the above identity we get
\begin{equation}\label{jacobithree}
\partial_{jk}g
=\operatorname{tr}\left((DH)^{-1}\partial_{jk}DH\right)
-\operatorname{tr}\left(
(DH)^{-1}(\partial_kDH)(DH)^{-1}(\partial_jDH)
\right).
\end{equation}

Using \eqref{jacobitwo} and \eqref{sumsecondjacobianentries}, this yields
\begin{equation}\label{jacobifour}
\left|\operatorname{tr}\left((DH)^{-1}\partial_{jk}DH\right)\right|
\le\frac{K+4E}{\kappa-2E},
\end{equation}
which estimates the first term in \eqref{jacobithree}.
To estimate the second term, observe that combining \eqref{jacobitwo} and \eqref{sumjacobianetries} with the Cauchy--Schwarz inequality applied to the trace gives
\[
\begin{aligned}
&\left|\operatorname{tr}\left(
(DH)^{-1}(\partial_kDH)(DH)^{-1}(\partial_jDH)
\right)\right|\\
&\quad\le
\left(\sum_{r,s}|[(DH)^{-1}\partial_kDH]_{r,s}|^2\right)^{1/2}
\left(\sum_{r,s}|[(DH)^{-1}\partial_jDH]_{r,s}|^2\right)^{1/2}\\
&\quad\le\frac1{(\kappa-2E)^2}
\left(\sum_{r,s}|(\partial_kDH)_{r,s}|^2\right)^{1/2}
\left(\sum_{r,s}|(\partial_jDH)_{r,s}|^2\right)^{1/2}\\
&\quad\le\frac{(K+E)^2+3E^2}{(\kappa-2E)^2},
\end{aligned}
\]
which proves \eqref{Eq:GeometryJacobianSecondRational}.

We now prove \eqref{zhengbian}. Fix $j\ne k$. We have
\begin{equation}\label{secondline}
\left|\operatorname{tr}\left((DH)^{-1}\partial_{jk}DH\right)\right|
\le\frac1{\kappa-2E}\sum_{r,s}|\partial_{jk}\partial_sH_r|
\le\frac{4E}{(\kappa-2E)N},
\end{equation}
where the last inequality follows since by \eqref{Eq:SecondDerivativeJacobianEntries} the only surviving terms have $r,s\in\{j,k\}$, and each has absolute value at most
$E/N$. Now observe that
\begin{equation}\label{theabovefi}
\left[(DH)^{-1}\partial_kDH\right]_{ij}
=\sum_l[(DH)^{-1}]_{il}\,\partial_k\partial_jH_l.
\end{equation}
The above formula, along with \eqref{Eq:InverseJacobianEntries}, the row-sum bound \eqref{jacobitwo} with $p=\infty$, and \eqref{Eq:CondHSecondDer}, gives
\[
\left|\left[(DH)^{-1}\partial_kDH\right]_{ij}\right|
\le
\begin{cases}
\displaystyle\frac{K+E}{\kappa-2E}, & i=j=k,\\[8pt]
\displaystyle
\frac{E}{(\kappa-2E)N}
\left(1+\frac{K+E}{\kappa-E}\right), & \text{otherwise}.
\end{cases}
\]
Set $A^{(r)}:=(DH)^{-1}(\partial_rDH)$.
Since
\[
\sum_{r,s}|A^{(j)}_{r,s}|
\le\|(DH)^{-1}\|_{1\to1}
\sum_{r,s}|(\partial_jDH)_{r,s}|
\le\frac{K+4E}{\kappa-2E},
\]
we obtain
\begin{align*}
&\left|\operatorname{tr}\left(
(DH)^{-1}(\partial_kDH)(DH)^{-1}(\partial_jDH)
\right)\right|\\
&\quad\le\sum_{r,s}|A^{(k)}_{r,s}A^{(j)}_{s,r}|\\
&\quad=
|A^{(k)}_{k,k}A^{(j)}_{k,k}|
+|A^{(k)}_{j,j}A^{(j)}_{j,j}|
+\sum_{\substack{r,s\\(r,s)\notin\{(k,k),(j,j)\}}}
|A^{(k)}_{r,s}A^{(j)}_{s,r}|\\
&\quad\le
\frac{E}{(\kappa-2E)N}
\left(1+\frac{K+E}{\kappa-E}\right)
\left(
|A^{(k)}_{k,k}|+|A^{(j)}_{j,j}|
+\sum_{\substack{r,s\\(r,s)\notin\{(k,k),(j,j)\}}}
|A^{(j)}_{s,r}|
\right)\\
&\quad\le
\frac{E}{(\kappa-2E)N}
\left(1+\frac{K+E}{\kappa-E}\right)
\left(
\frac{K+E}{\kappa-2E}
+\frac{K+4E}{\kappa-2E}
\right)\\
&\quad=
\frac{E(2K+5E)}{(\kappa-2E)^2N}
\left(1+\frac{K+E}{\kappa-E}\right).
\end{align*}
Together with \eqref{jacobithree} and \eqref{secondline},
this proves \eqref{zhengbian}.
\end{proof}

\subsection{Degree-based decomposition}
In this subsection we show that $H= \mathsf D\circ \Theta$, where $\mathsf D$ is a linear map and $\Theta$ is a diffeomorphism of $\mathbb{T}^N.$ This decomposition will be used in the next subsection to prove Proposition \ref{Prop:LipschitzReg2}.
\begin{lemma}\label{C567:lem:geometry}
Under the standing assumptions on $H$, the degree matrix of $H$ is diagonal,
$\mathsf D=\operatorname{diag}(d_1,\ldots,d_N)$, with
$|d_i|\ge\kappa-E$. There is a diffeomorphism $\Theta$,
homotopic to the identity, such that $H=\mathsf D\circ\Theta$.
Write $\psi=\Theta^{-1}$ . Then we have
\begin{equation}\label{Eq:GeometryInverseRational}
\|D\psi\|_{p\to p}
\le\frac{\kappa-E+(K+E)/4}{\kappa-2E},
\qquad p\in\{1,\infty\},
\end{equation}
and
\begin{equation}\label{Eq:GeometryInverseOffDiagonal}
|\partial_j\psi_r|
\le\frac{E}{(\kappa-2E)N}
\left(1+\frac{K+E}{4(\kappa-E)}\right),
\qquad r\ne j.
\end{equation}
Furthermore, for $j\ne k$, we also get
\begin{equation}\label{Eq:GeometryOverlapRational}
\sup_w\sum_r|\partial_j\psi_r(w)\,\partial_k\psi_r(w)|
\le\frac{2E}{(\kappa-E)N}
\left(\frac{\kappa-E+(K+E)/4}{\kappa-2E}\right)^2,
\end{equation}
and 
\begin{equation}\label{Eq:GeometrySecondMixedRational}
\sup_w\sum_r|\partial_{jk}\psi_r(w)|
\le\frac{2E}{(\kappa-2E)N}
\left(1+\frac{K+2E}{\kappa-E}\right)
\left(\frac{\kappa-E+(K+E)/4}{\kappa-2E}\right)^2,
\end{equation}
and for every $j$,
\begin{equation}\label{Eq:GeometrySecondDiagonalRational}
\sup_w\sum_r|\partial_{jj}\psi_r(w)|
\le
\frac{K+2E}{\kappa-2E}
\left(
\frac{\kappa-E+(K+E)/4}{\kappa-2E}
\right)^2.
\end{equation}

\end{lemma}

\begin{proof}
\noindent\emph{Step 1: the degree matrix and the factorization.}
Let $\widetilde H$ be a lift of $H$ to $\mathbb R^N$.
For $i\ne j$, the difference
$\widetilde H_i(x+e_j)-\widetilde H_i(x)$ is integer-valued
and has absolute value at most $E/N<1$, hence it vanishes. Analogously,
$\widetilde H_i(x+e_i)-\widetilde H_i(x)$ is integer-valued and
continuous in $x$, hence equal to a constant
\[
d_i:=\widetilde H_i(x+e_i)-\widetilde H_i(x)\in\mathbb Z,
\]
with $|d_i|\ge\kappa-E$. It follows that the degree matrix of $H$
is $\mathsf D=\operatorname{diag}(d_1,\ldots,d_N)$.
Define $\Theta\colon\mathbb T^N\to\mathbb T^N$ through its lift
$\widetilde\Theta\colon\mathbb R^N\to\mathbb R^N$, where
\[
\widetilde\Theta_i=\frac{\widetilde H_i}{d_i},
\qquad 1\le i\le N.
\]
Then $\widetilde\Theta(x+k)=\widetilde\Theta(x)+k$ for every
$k\in\mathbb Z^N$, so $\Theta$ is well-defined and homotopic
to the identity. Since $\Theta$ is also a local diffeomorphism, it is a
diffeomorphism of $\mathbb T^N$. By construction,
$H=\mathsf D\circ\Theta$.
\par
\smallskip
\noindent\emph{Step 2: estimates for $D\psi$.}
We prove \eqref{Eq:GeometryInverseRational} and
\eqref{Eq:GeometryInverseOffDiagonal}. Write $DH=A+R$, with
$A=\operatorname{diag}(\partial_1H_1,\ldots,\partial_NH_N)$.
Since $H\circ\psi=\mathsf D$, we have
\begin{equation}\label{invratone}
D\psi(x)
=\bigl[DH(\psi(x))\bigr]^{-1}\mathsf D
=\bigl[I+(A^{-1}R)(\psi(x))\bigr]^{-1}
A(\psi(x))^{-1}\mathsf D.
\end{equation}
The matrix $A^{-1}\mathsf D$ is diagonal, with entries
$d_i/\partial_iH_i$. Distortion bounds give
\[
\left|\frac{d_i}{\partial_iH_i}\right|
\le1+\frac{K+E}{4(\kappa-E)}.
\]
Combining this estimate with \eqref{invratone} and a Neumann series argument using
\eqref{diagonalentriesar-1} proves \eqref{Eq:GeometryInverseRational}.
Similarly, \eqref{Eq:GeometryInverseOffDiagonal} follows from \eqref{invratone} and a Neumann series argument using \eqref{arentries}.
\par
\smallskip

\noindent\emph{Step 3: the remaining estimates for $\psi$.}
We first prove \eqref{Eq:GeometryOverlapRational} and
\eqref{Eq:GeometrySecondMixedRational}. Fix $j\ne k$.
Separating the term $r=k$, we obtain
\[
\begin{aligned}
\sum_r|\partial_j\psi_r\,\partial_k\psi_r|
&=\sum_{r\ne k}|\partial_j\psi_r\,\partial_k\psi_r|
+|\partial_j\psi_k\,\partial_k\psi_k|\\
&\le\frac{E}{(\kappa-2E)N}
\left(1+\frac{K+E}{4(\kappa-E)}\right)
\left(\sum_{r\ne k}|\partial_j\psi_r|+|\partial_k\psi_k|\right)\\
&\le\frac{2E}{(\kappa-E)N}
\left(\frac{\kappa-E+(K+E)/4}{\kappa-2E}\right)^2.
\end{aligned}
\]
Here we used \eqref{Eq:GeometryInverseRational} and
\eqref{Eq:GeometryInverseOffDiagonal}. Taking the supremum
proves \eqref{Eq:GeometryOverlapRational}.
Since $H\circ\psi=\mathsf D$ and $\mathsf D$ is linear, we have $\partial_j\psi=d_j[DH(\psi)]^{-1}e_j$. Differentiating once more yields
\begin{align*}
0
&=DH(\psi(y))\,\partial_{jk}\psi(y)
+\partial_k[DH(\psi(y))]\,\partial_j\psi(y)\\
&=DH(\psi(y))\,\partial_{jk}\psi(y)
+\sum_q(\partial_qDH)(\psi(y))\,\partial_k\psi_q(y)\,\partial_j\psi(y)\\
&=DH(\psi(y))\,\partial_{jk}\psi(y)
+\sum_{p,q}(\partial_{pq}H)(\psi(y))\,
\partial_j\psi_p(y)\,\partial_k\psi_q(y),
\end{align*}
where
\[
\partial_{pq}H=
\begin{pmatrix}
\partial_{pq}H_1\\
\vdots\\
\partial_{pq}H_N
\end{pmatrix}
\]
is the $p$-th column of $\partial_qDH$. Consequently,
\begin{equation}\label{Eq:GeometrySecondInverseIdentity}
\partial_{jk}\psi
=-[DH(\psi)]^{-1}
\sum_{p,q}(\partial_{pq}H)(\psi)\,
\partial_j\psi_p\,\partial_k\psi_q.
\end{equation}
Furthermore, formula \eqref{Eq:CondHSecondDer} gives 
$\sum_r|\partial_{pp}H_r|\le K+2E$ and, if $p\neq q,$ it gives $\sum_r|\partial_{pq}H_r|\le\frac{2E}{N}.$ Splitting the sum in \eqref{Eq:GeometrySecondInverseIdentity}
according to whether $p=q$, we then obtain
\[
\begin{aligned}
\sum_r|\partial_{jk}\psi_r|
&\le\frac{K+2E}{\kappa-2E}
\sum_p|\partial_j\psi_p\,\partial_k\psi_p|\\
&\quad+\frac{2E}{(\kappa-2E)N}
\left(\sum_p|\partial_j\psi_p|\right)
\left(\sum_q|\partial_k\psi_q|\right)\\
&\le\frac{2E}{(\kappa-2E)N}
\left(1+\frac{K+2E}{\kappa-E}\right)
\left(\frac{\kappa-E+(K+E)/4}{\kappa-2E}\right)^2.
\end{aligned}
\]
The last inequality follows from
\eqref{Eq:GeometryInverseRational} and
\eqref{Eq:GeometryOverlapRational}.
Taking the supremum proves \eqref{Eq:GeometrySecondMixedRational}.

Finally, we prove \eqref{Eq:GeometrySecondDiagonalRational}.
The identity \eqref{Eq:GeometrySecondInverseIdentity}
also holds for $j=k$. Since, by  \eqref{Eq:CondHSecondDer}, $\sum_r|\partial_{pq}H_r|\le K+2E$, for all $p,q$, formulas
\eqref{jacobitwo} and
\eqref{Eq:GeometryInverseRational} give
\[
\begin{aligned}
\sum_r|\partial_{jj}\psi_r(w)|
&\le
\frac{1}{\kappa-2E}
\sum_{p,q}
\left(\sum_r|(\partial_{pq}H_r)(\psi(w))|\right)
|\partial_j\psi_p(w)|\,|\partial_j\psi_q(w)|
\\
&\le
\frac{K+2E}{\kappa-2E}
\left(\sum_p|\partial_j\psi_p(w)|\right)^2
\\
&\le
\frac{K+2E}{\kappa-2E}
\left(
\frac{\kappa-E+(K+E)/4}{\kappa-2E}
\right)^2.
\end{aligned}
\]
Taking the supremum over $w$ proves
\eqref{Eq:GeometrySecondDiagonalRational}.
\end{proof}
\subsection{The invariant class}
In this section we conclude the proof of Proposition \ref{Prop:LipschitzReg2}. Recall that by Lemma \ref{C567:lem:geometry}, $H=\mathsf D\circ \Theta$, where $\mathsf D$ is linear and $\Theta$ is a diffeomorphism of $\mathbb{T}^N$. We first study the regularity of the pushforwards via $\Theta$ of densities in $\mathcal{Y}_N(a,M,\gamma)$.

\begin{theorem}\label{thm:rectified-density}
Under the assumptions of Lemma~\ref{C567:lem:geometry},
let $\eta\in\mathcal Y_N(a,M,\gamma)$, where $M\ge1,$ $0\le\gamma\le1$ and $\mathcal Y_N$ has been defined in \eqref{Eq:InvariantClass}. Put $\rho=\Theta_*\eta$ and define
\begin{equation}\label{Eq:LE}
\begin{aligned}
L_E&:=\frac{\kappa-E+(K+E)/4}{\kappa-2E},\\
J_E&:=\frac{K+4E}{\kappa-2E},\\
Q_E&:=\frac{K+4E}{\kappa-2E}
+\frac{(K+E)^2+3E^2}{(\kappa-2E)^2},\\
R_E&:=E\left[
\frac4{\kappa-2E}
+\frac{2K+5E}{(\kappa-2E)^2}
\left(1+\frac{K+E}{\kappa-E}\right)
\right].
\end{aligned}
\end{equation}
Then $\rho\in\mathcal Y_N(a_\Theta,M_\Theta,\gamma_\Theta)$, with
\begin{align*}
a_\Theta&:=L_E(a+J_E),\\
M_\Theta&:= L_E^2\left[
M+Q_E+\frac{\gamma+R_E}{N}
+\frac{K+2E}{\kappa-2E}(a+J_E)
\right],\\
\gamma_\Theta&:=L_E^2
\left[
\gamma+R_E
+\frac{2E}{\kappa-E}(M+Q_E)
+\frac{2E}{\kappa-2E}(a+J_E)
\left(1+\frac{K+2E}{\kappa-E}\right)
\right].
\end{align*}
\end{theorem}
The proof of this Theorem is deferred to the end of this subsection. We now study the regularity of the pushforward via $\mathsf D$ of densities in $\mathcal{Y}_N(a,M,\gamma)$.

\begin{theorem}\label{thm:diagonal-pushforward}
Let $\mathsf D=\operatorname{diag}(d_1,\ldots,d_N)$, where
$d_j\in\mathbb Z\setminus\{0\}$, and put
$d_*:=\min_{1\le j\le N}|d_j|$.
Suppose that $\rho\in\mathcal Y_N(a,M,\gamma)$, where $0\le\gamma<16$.
Then
\[
\zeta:=\mathsf D_*\rho\in\mathcal Y_N(a_D,M_D,\gamma_D),
\]
where
\begin{align*}
a_D&:=\frac{a}{d_*},\\
M_D&:=\frac{M+a^2}{d_*^2},\\
\gamma_D&:=\frac{\gamma}{d_*^2}
\frac{(1+a/4)^2}{1-\gamma/16}.
\end{align*}
\end{theorem}
Its proof is also deferred to the end of this subsection. We are now ready to prove Proposition \ref{Prop:LipschitzReg2}.

\begin{proof}[Proof of Proposition \ref{Prop:LipschitzReg2}]
Fix $\kappa,K$ as in \eqref{Eq:MainExpansion} and let $\eta\in\mc Y_N(a,M,\gamma)$. Applying
Theorems~\ref{thm:rectified-density}
and~\ref{thm:diagonal-pushforward}, we obtain
$H_*\eta\in\mc Y_N(a',M',\gamma')$ with
\begin{align*}
a'&=\frac{L_E(a+J_E)}{\kappa-E},\\
M'&=\frac{L_E^2}{(\kappa-E)^2}
\left[
M+Q_E+\frac{\gamma+R_E}{N}
+\frac{K+2E}{\kappa-2E}(a+J_E)
+(a+J_E)^2
\right],\\
\gamma'&=
\frac{\left(1+\frac K{4\kappa}\right)^2}{\kappa^2}
\left[
1+\frac14\left(1+\frac K{4\kappa}\right)
\left(a+\frac K\kappa\right)
\right]^2\gamma
+C(a,M)(E+\gamma^2),
\end{align*}
where $C=C(a,M)$ is a constant that, after fixing
$a,M,$ depends only on $K,\kappa$ when $E$ is small enough.
By \eqref{Eq:LE},
\[
a'=\frac{L_E(a+J_E)}{\kappa-E}<a
\quad\Longleftrightarrow\quad
a>\frac{L_EJ_E}{\kappa-E-L_E},
\]
provided $\kappa-E-L_E>0$, and, as $E\to0$, the right-hand side
converges to
\[
\frac{K\left(1+\frac K{4\kappa}\right)}
{\kappa\left(\kappa-1-\frac K{4\kappa}\right)},
\]
where $\kappa-1-\frac K{4\kappa}>0$ by \eqref{Eq:MainExpansion}.
Writing $L_0:=1+\frac K{4\kappa}$, a direct computation shows that the
interval in \eqref{upperboundona} below is nonempty if and only if
$KL_0^2<4(\kappa-L_0)^2$, that is, if and only if
\eqref{Eq:MainExpansion} holds. We choose
\begin{equation}\label{upperboundona}
\frac{K\left(1+\frac K{4\kappa}\right)}
{\kappa\left(\kappa-1-\frac K{4\kappa}\right)}
<a<
\frac{4\kappa}{\left(1+\frac K{4\kappa}\right)^2}
-\frac4{1+\frac K{4\kappa}}-\frac K\kappa,
\end{equation}
which gives $a'<a$ for $E$ small enough, by continuity in $E$. The upper bound on $a$ will be used later.
We also have  $M'<M$ if and only if
\[
M>
\frac{L_E^2}{(\kappa-E)^2-L_E^2}
\left[
Q_E+\frac{\gamma+R_E}{N}
+\frac{K+2E}{\kappa-2E}(a+J_E)
+(a+J_E)^2
\right].
\]
For the condition $M\ge M',$ it is then sufficient to choose $M\ge1$ with
$$
M>\frac{L_0^2}{\kappa^2-L_0^2}\,R,\qquad
R:=\lim_{E\to 0}\left[
Q_E+\frac{1+R_E}{2}
+\frac{K+2E}{\kappa-2E}(a+J_E)
+(a+J_E)^2
\right],
$$
provided we choose $\gamma<1$ below and $E$ is small enough. Here we
used $(\gamma+R_E)/N\le(1+R_E)/2$ for $N\ge2$, $L_E\to L_0$ as
$E\to0$, and $\kappa>L_0$, which follows from
\eqref{Eq:MainExpansion}.
For $\gamma'<\gamma$, we need
\begin{align*}
\gamma'&\le
\frac{\left(1+\frac K{4\kappa}\right)^2}{\kappa^2}
\left[
1+\frac14\left(1+\frac K{4\kappa}\right)
\left(a+\frac K\kappa\right)
\right]^2\gamma
+C(E+\gamma^2),
\end{align*}
where $C=C(a,M)$ for our previous choices of $a,M$.
By the upper bound on $a$ in \eqref{upperboundona}, the coefficient multiplying $\gamma$ is strictly smaller than $1$. Choosing
$\gamma=\gamma_E$ as a sufficiently large multiple of $E$,
and then decreasing $E_0<1$, gives
$\gamma'\le\gamma$ and $\gamma\le1/8<1$. 
\end{proof}

\begin{proof}[Proof of Theorem \ref{thm:rectified-density}]
Take $\eta\in\mathcal Y_N(a,M,\gamma)$ and put
$\rho=\Theta_*\eta$. Then
\[
\rho(w)=\eta(\psi(w))|\det D\psi(w)|.
\]
Differentiating the identity $H\circ\psi=\mathsf D$ and taking
determinants gives
\[
\log|\det D\psi(w)|
=
\log|\det \mathsf D|-\log|\det DH(\psi(w))|.
\]
Let $V=\log\rho$. Then
\begin{equation}\label{fomrulaforV}
V=
\bigl(\log\eta-\log|\det DH|\bigr)\circ\psi
+\log|\det \mathsf D|.
\end{equation}
Differentiating once, we obtain
\begin{equation}\label{diffonce}
\partial_jV(w)
=
\sum_r
\bigl(
\partial_r(\log\eta)
-\partial_r\left(\log|\det DH|\right)
\bigr)(\psi(w))
\,\partial_j\psi_r(w).
\end{equation}
For the second derivative we get
\begin{equation}\label{difftwice}
\begin{aligned}
\partial_{jk}V(w)
={}&
\sum_{r,s}
\bigl(
\partial_{rs}(\log\eta)
-\partial_{rs}\left(\log|\det DH|\right)
\bigr)(\psi(w))\,
\partial_j\psi_r(w)\,
\partial_k\psi_s(w)
\\
&+
\sum_r
\bigl(
\partial_r(\log\eta)
-\partial_r\left(\log|\det DH|\right)
\bigr)(\psi(w))\,
\partial_{jk}\psi_r(w).
\end{aligned}
\end{equation}

Using \eqref{diffonce}, \eqref{Eq:GeometryInverseRational},
and \eqref{Eq:GeometryJacobianRational}, we obtain
\begin{align}
\|\partial_jV\|_\infty
&\le
\left(
a+\frac{K+4E}{\kappa-2E}
\right)
\sup_w\sum_r|\partial_j\psi_r(w)|
\nonumber\\
&\le
\left(
a+\frac{K+4E}{\kappa-2E}
\right)
\frac{\kappa-E+(K+E)/4}{\kappa-2E}.
\label{Eq:RectifiedFirst}
\end{align}
Write
\[
|\partial_{jk}V(w)|\le I_1(w)+I_2(w)+I_3(w),
\]
where
\begin{align*}
I_1(w)
&:=
\sum_{r\ne s}
\left|
\bigl(
\partial_{rs}(\log\eta)
-\partial_{rs}(\log|\det DH|)
\bigr)(\psi(w))
\right|
|\partial_j\psi_r(w)|\,|\partial_k\psi_s(w)|,
\\
I_2(w)
&:=
\sum_r
\left|
\bigl(
\partial_{rr}(\log\eta)
-\partial_{rr}(\log|\det DH|)
\bigr)(\psi(w))
\right|
|\partial_j\psi_r(w)\,\partial_k\psi_r(w)|,
\\
I_3(w)
&:=
\sum_r
\left|
\bigl(
\partial_r(\log\eta)
-\partial_r(\log|\det DH|)
\bigr)(\psi(w))
\right|
|\partial_{jk}\psi_r(w)|.
\end{align*}
We first treat the case $j\neq k$. For the off-diagonal Hessian terms, \eqref{zhengbian}
and \eqref{Eq:GeometryInverseRational} give
\begin{align*}
I_1(w)
&\le
\left[
\frac{\gamma}{N}
+\frac{E}{N}
\left(
\frac4{\kappa-2E}
+\frac{2K+5E}{(\kappa-2E)^2}
\left(1+\frac{K+E}{\kappa-E}\right)
\right)
\right]
\sum_{r\ne s}
|\partial_j\psi_r(w)|\,|\partial_k\psi_s(w)|
\\
&\le
\frac1N
\left[
\gamma+
E\left(
\frac4{\kappa-2E}
+\frac{2K+5E}{(\kappa-2E)^2}
\left(1+\frac{K+E}{\kappa-E}\right)
\right)
\right]
\\
&\qquad\times\left(\sum_r|\partial_j\psi_r(w)|\right)
\left(\sum_s|\partial_k\psi_s(w)|\right)
\\
&\le
\frac1N
\left[
\gamma+
E\left(
\frac4{\kappa-2E}
+\frac{2K+5E}{(\kappa-2E)^2}
\left(1+\frac{K+E}{\kappa-E}\right)
\right)
\right]
\\
&\qquad\times\left(
\frac{\kappa-E+(K+E)/4}{\kappa-2E}
\right)^2.
\end{align*}
For the diagonal Hessian terms,
\eqref{Eq:GeometryJacobianSecondRational}
and \eqref{Eq:GeometryOverlapRational} give
\begin{align*}
I_2(w)
&\le
\left(
M+\frac{K+4E}{\kappa-2E}
+\frac{(K+E)^2+3E^2}{(\kappa-2E)^2}
\right)
\sum_r|\partial_j\psi_r(w)\,\partial_k\psi_r(w)|
\\
&\le
\left(
M+\frac{K+4E}{\kappa-2E}
+\frac{(K+E)^2+3E^2}{(\kappa-2E)^2}
\right)
\frac{2E}{(\kappa-E)N}
\left(
\frac{\kappa-E+(K+E)/4}{\kappa-2E}
\right)^2.
\end{align*}
Finally, \eqref{Eq:GeometryJacobianRational}
and \eqref{Eq:GeometrySecondMixedRational} give
\begin{align*}
I_3(w)
&\le
\left(a+\frac{K+4E}{\kappa-2E}\right)
\sum_r|\partial_{jk}\psi_r(w)|
\\
&\le
\left(a+\frac{K+4E}{\kappa-2E}\right)
\frac{2E}{(\kappa-2E)N}
\left(1+\frac{K+2E}{\kappa-E}\right)
\left(
\frac{\kappa-E+(K+E)/4}{\kappa-2E}
\right)^2.
\end{align*}
For $j=k$, the estimate for $I_1$ remains unchanged.
For $I_2$, we use \eqref{Eq:GeometryInverseRational} and
\eqref{Eq:GeometryJacobianSecondRational} to obtain
\[
I_2(w)
\le
\left(
M+\frac{K+4E}{\kappa-2E}
+\frac{(K+E)^2+3E^2}{(\kappa-2E)^2}
\right)
\left(
\frac{\kappa-E+(K+E)/4}{\kappa-2E}
\right)^2.
\]
For $I_3$ we use \eqref{Eq:GeometryJacobianRational}
and \eqref{Eq:GeometrySecondDiagonalRational} to obtain
\begin{align*}
I_3(w)
&\le
\left(a+\frac{K+4E}{\kappa-2E}\right)
\sum_r|\partial_{jj}\psi_r(w)|
\\
&\le
\left(a+\frac{K+4E}{\kappa-2E}\right)
\frac{K+2E}{\kappa-2E}
\left(
\frac{\kappa-E+(K+E)/4}{\kappa-2E}
\right)^2.
\end{align*}
\end{proof}

\begin{proof}[Proof of Theorem \ref{thm:diagonal-pushforward}]
Write $V=\log\rho$ and $\zeta=\mathsf D_*\rho$. Define the branch-label sets
\[
\Sigma_r:=\{0,\ldots,|d_r|-1\},\qquad
\Sigma:=\prod_{r=1}^N\Sigma_r.
\]
For $\sigma\in\Sigma$ and $y\in\T^N$, let
\[
z_\sigma(y)_r=\frac{y_r+\sigma_r}{d_r}\pmod1.
\]
These are local inverse branches; all derivative identities below are
local in $y$ and agree on overlapping charts. Then
\[
\zeta(y)=\frac1{|\det\mathsf D|}\sum_{\sigma\in\Sigma}e^{V(z_\sigma(y))}.
\]
Set
\begin{align*}
a_\sigma(y)&:=e^{V(z_\sigma(y))},&
S(y)&:=\sum_{\tau\in\Sigma}a_\tau(y),&
\pi_y(\sigma)&:=\frac{a_\sigma(y)}{S(y)}.
\end{align*}
For functions $g,h:\Sigma\to\R$, define
\[
\mathbb E_{\pi_y}[g]:=\sum_{\sigma\in\Sigma}\pi_y(\sigma)g(\sigma),
\qquad
\operatorname{Cov}_{\pi_y}(g,h)
:=\mathbb E_{\pi_y}[gh]-\mathbb E_{\pi_y}[g]\mathbb E_{\pi_y}[h].
\]
For the spatial derivatives of $V$, use the pullbacks
\[
V_{j,y}(\sigma):=(\partial_jV)(z_\sigma(y)),\qquad
V_{jk,y}(\sigma):=(\partial_{jk}V)(z_\sigma(y)).
\]
Differentiation gives
\[
\partial_jS(y)=\frac1{d_j}\sum_\sigma a_\sigma(y)V_{j,y}(\sigma),
\]
and hence
\begin{equation}\label{firstderforzed}
\partial_j\log\zeta(y)=\frac1{d_j}\mathbb E_{\pi_y}[V_{j,y}].
\end{equation}
Moreover,
\[
\partial_k\pi_y(\sigma)
=\frac{\pi_y(\sigma)}{d_k}
\bigl(V_{k,y}(\sigma)-\mathbb E_{\pi_y}[V_{k,y}]\bigr),\qquad
\partial_kV_{j,y}(\sigma)=\frac1{d_k}V_{jk,y}(\sigma).
\]
Differentiating \eqref{firstderforzed} therefore yields
\begin{equation}\label{secondorderzet}
\partial_{jk}\log\zeta(y)
=\frac{\mathbb E_{\pi_y}[V_{jk,y}]
+\operatorname{Cov}_{\pi_y}(V_{j,y},V_{k,y})}{d_jd_k}.
\end{equation}
Since $|V_{j,y}|\le a$ and $|V_{jj,y}|\le M$,
\[
\|\partial_j\log\zeta\|_\infty\le\frac a{d_*}=a_D,
\qquad
\|\partial_{jj}\log\zeta\|_\infty
\le\frac{M+a^2}{d_*^2}=M_D.
\]
For $j\ne k$, it remains to estimate
\begin{equation}\label{covarianceterm}
|\partial_{jk}\log\zeta(y)|
\le\frac1{d_*^2}\left[
\frac\gamma N+|\operatorname{Cov}_{\pi_y}(V_{j,y},V_{k,y})|
\right].
\end{equation}

For $g:\Sigma\to\R$, define its coordinate oscillation by
\[
\delta_r g:=
\max_{\substack{\sigma,\tau\in\Sigma\\
\sigma_\ell=\tau_\ell\ (\ell\ne r)}}|g(\sigma)-g(\tau)|.
\]
Write $\pi_{y,r\mid-r}(\cdot\mid\sigma_{-r})$ for the conditional
law of the $r$-th label, given the other labels $\sigma_{-r}$.
We use the following covariance estimate.

\begin{theorem}[Föllmer's covariance estimate, {\cite{follmer1982covariance}}]
\label{thm:covariance}
Let $C$ be a nonnegative matrix with $C_{rr}=0$ such that, for $r\ne t$,
\[
d_{\mathrm{TV}}\bigl(
\pi_{y,r\mid-r}(\cdot\mid\sigma_{-r}),
\pi_{y,r\mid-r}(\cdot\mid\tau_{-r})\bigr)\le C_{rt}
\]
for every $y$, whenever $\sigma,\tau$ differ only in coordinate $t$.
If $\|C\|_{\infty\to\infty}<1$, then, for $g,h:\Sigma\to\R$,
\[
|\operatorname{Cov}_{\pi_y}(g,h)|
\le\frac14\sum_{r,t=1}^N
\delta_r g\,[(I-C)^{-1}]_{rt}\,\delta_t h.
\]
\end{theorem}

To bound the influence matrix $C$, recall that for positive probability
mass functions $p,q$ on $\Sigma_r$,
\begin{equation}\label{followinggeneralfact}
d_{\mathrm{TV}}(p,q)
\le\frac14\operatorname{osc}_{\alpha\in\Sigma_r}
\log\frac{p(\alpha)}{q(\alpha)}.
\end{equation}
Fix $y$, $r\ne t$, and labels $\sigma,\tau$ that differ only at $t$.
Then
\[
\pi_{y,r\mid-r}(\alpha\mid\sigma_{-r})
=\frac{e^{V(z_{(\alpha,\sigma_{-r})}(y))}}{Z_\sigma},
\qquad
Z_\sigma:=\sum_{\beta\in\Sigma_r}
e^{V(z_{(\beta,\sigma_{-r})}(y))},
\]
and similarly for $\tau$. The normalizing factors disappear under oscillation:
\begin{equation}\label{coolosctrick}
\operatorname{osc}_{\alpha\in\Sigma_r}
\log\frac{\pi_{y,r\mid-r}(\alpha\mid\sigma_{-r})}
{\pi_{y,r\mid-r}(\alpha\mid\tau_{-r})}
=\operatorname{osc}_{\alpha\in\Sigma_r}
\bigl[V(z_{(\alpha,\sigma_{-r})}(y))
-V(z_{(\alpha,\tau_{-r})}(y))\bigr].
\end{equation}
Fix $\alpha,\beta\in\Sigma_r$. Let $u_\alpha,u_\beta$ be the corresponding
$r$-th preimage coordinates, let $v_\sigma,v_\tau$ be the $t$-th preimage
coordinates, and let $w$ collect the common coordinates outside $\{r,t\}$.
Writing $V(u,v,w)$ with $u,v$ in slots $r,t$, respectively, gives
\begin{align*}
&V(u_\alpha,v_\sigma,w)-V(u_\alpha,v_\tau,w)
-V(u_\beta,v_\sigma,w)+V(u_\beta,v_\tau,w)\\
&\qquad=\int_{u_\beta}^{u_\alpha}\int_{v_\tau}^{v_\sigma}
\partial_{rt}V(u,v,w)\,dv\,du.
\end{align*}
Integrate along shortest oriented circle arcs, each of length at most $1/2$.
The absolute value is therefore at most $\gamma/(4N)$.
By \eqref{followinggeneralfact}--\eqref{coolosctrick}, we can choose
\begin{equation}\label{estimates on c}
C_{rt}=\begin{cases}0,&r=t,\\ \gamma/(16N),&r\ne t.\end{cases}
\end{equation}
Since $\gamma<16$, its row-sum norm satisfies
$\|C\|_{\infty\to\infty}\le\gamma/16<1$.
For $n\ge1$, $(C^n)_{rt}\le N^{-1}(\gamma/16)^n$, and hence
\[
[(I-C)^{-1}]_{rt}
\le\mathbf1_{r=t}+\frac{\gamma}{16N(1-\gamma/16)}.
\]
Shortest-arc integration also gives
\[
\delta_rV_{j,y}\le
\begin{cases}2a,&r=j,\\ \gamma/(2N),&r\ne j.\end{cases}
\]
Consequently, for $j\ne k$,
\[
\sum_r\delta_rV_{j,y}\le2a+\gamma/2,\qquad
\sum_r\delta_rV_{j,y}\,\delta_rV_{k,y}
\le\frac{2a\gamma+\gamma^2/4}{N}.
\]
Theorem~\ref{thm:covariance} now gives
\[
|\operatorname{Cov}_{\pi_y}(V_{j,y},V_{k,y})|
\le\frac{a\gamma}{2N}+\frac{\gamma^2}{16N}
+\frac{\gamma(2a+\gamma/2)^2}{64N(1-\gamma/16)}.
\]
Substituting into \eqref{covarianceterm} and collecting terms yields
\begin{align*}
|\partial_{jk}\log\zeta(y)|
&\le\frac{\gamma}{Nd_*^2}
\left[1+\frac a2+\frac\gamma{16}
+\frac{(2a+\gamma/2)^2}{64(1-\gamma/16)}\right]\\
&=\frac{\gamma}{Nd_*^2}\frac{(1+a/4)^2}{1-\gamma/16}
=\frac{\gamma_D}{N}.
\end{align*}
This proves the theorem.
\end{proof}

\section{Stability Estimate}\label{stabbing}
In this section we prove Lemma \ref{Lem:S6TransferEstimates}.
\begin{proof}
Let $T$ be as in \eqref{tands} and let $\mathcal{L}_T$ denote its associated transfer operator. By the first-derivative formula in \cite[Equation~(5)]{Galatolo2015}, for every $v\in C^1(\mathbb T)$,
\[
(\mathcal L_Tv)'=\mathcal L_T\left(\frac{v'}{T'}-\frac{vT''}{(T')^2}\right),
\]
which implies 
\begin{equation}\label{firstorder}
\|(\mathcal L_Tv)'\|_1\le\frac1s\|v'\|_1+\frac D{s^2}\|v\|_1
\end{equation}
by \eqref{tands}. Formula \eqref{samemap} follows by choosing $v=p-q$, which has zero integral and therefore satisfies $\|v\|_1\le\|v'\|_1$.

We now prove \eqref{Eq:S6TransferPerturbation}. Fix $s,D,a,R>0$ and $T,S$ as in \eqref{tands} and $p$ as in \eqref{formulaforp}, and let $\mathcal L_S$ and $\mathcal L_T$ denote their respective transfer operators.
For the fixed lifts, put
\[
h:=\widetilde S-\widetilde T,\qquad
T_r:=(\widetilde T+rh)\bmod1,\qquad
\mathcal L_r:=\mathcal L_{T_r}\quad(0\le r\le1).
\]
The derivatives of $T$ and $S$ have the same sign, since their
degrees agree, so $\inf|T_r'|\ge s$.
We claim that

\begin{equation}\label{wefirstclaim}
d_{\partial}(T_*p,S_*p)\le 
\int_0^1\|[\mathcal L_r(ph)]''\|_1\,dr,
\end{equation}
Indeed, if the above holds, applying the first-derivative formula in 
\cite[Equation~(5)]{Galatolo2015} twice,  for every $v\in C^2(\mathbb T)$ we obtain
\[
(\mathcal L_r v)''
=
\mathcal L_r\left(
\frac{v''}{(T_r')^2}
-\frac{3v'T_r''}{(T_r')^3}
-\frac{vT_r'''}{(T_r')^3}
+\frac{3v(T_r'')^2}{(T_r')^4}
\right),
\]
which, by \eqref{tands}, implies
\begin{equation}\label{mathcalr}
\|(\mathcal L_r v)''\|_1
\le
\frac1{s^2}\|v''\|_1
+\frac{3D}{s^3}\|v'\|_1
+\left(\frac D{s^3}+\frac{3D^2}{s^4}\right)\|v\|_1.    
\end{equation}
Set $v:=ph$. Then, by \eqref{formulaforp} and the product rule, we have
$\|v\|_1\le\|h\|_{C^2}$,
$\|v'\|_1\le(a+1)\|h\|_{C^2}$ and
$\|v''\|_1\le(R+2a+1)\|h\|_{C^2}$. The above estimates, combined with \eqref{mathcalr} and \eqref{wefirstclaim}, give \eqref{Eq:S6TransferPerturbation}.

It remains to prove \eqref{wefirstclaim}. Note that for every smooth test function $\varphi$, we have
\[
\begin{aligned}
\frac{d}{dr}\int_{\mathbb T}\varphi(y)\mathcal L_rp(y)\,dy
&=
\int_{\mathbb T}\varphi'(T_r(x))h(x)p(x)\,dx\\
&=
\int_{\mathbb T}\varphi'(y)\mathcal L_r(ph)(y)\,dy\\
&=
-\int_{\mathbb T}\varphi(y)[\mathcal L_r(ph)]'(y)\,dy.
\end{aligned}
\]
Thus $\partial_r(\mathcal L_rp)
=-[\mathcal L_r(ph)]'.$ Differentiating in the spatial variable and integrating in $r$
gives
\[
(S_*p)'-(T_*p)'
=-\int_0^1[\mathcal L_r(ph)]''\,dr,
\]
which implies \eqref{wefirstclaim}.
\end{proof}

\section{Comparison to product}\label{stabbings}
In this appendix we prove Lemma \ref{Lem:S6Marginalisation} and Proposition \ref{Prop:OneStepTV}.

\begin{proof}[Proof of Lemma \ref{Lem:S6Marginalisation}]
For simplicity, assume that $S=\{1,\ldots,k\}$ and write
\[
\mu_{j,\ldots,k}:=\Pi_{[j,k]}\mu,
\qquad
\mu_i:=\Pi_i\mu.
\]
A telescoping argument gives
\[
\mu_{1,\ldots,k}-\bigotimes_{i=1}^k\mu_i
=
\sum_{j=1}^{k-1}
\left(\bigotimes_{i=1}^{j-1}\mu_i\right)
\otimes
\left(
\mu_{j,\ldots,k}
-\mu_j\otimes\mu_{j+1,\ldots,k}
\right).
\]
Therefore,
\begin{equation}\label{Eq:TelescopingFactorisation}
\left\|
\mu_{1,\ldots,k}
-\bigotimes_{i=1}^k\mu_i
\right\|_{\mathrm{TV}}
\le
\sum_{j=1}^{k-1}
\left\|
\mu_{j,\ldots,k}
-\mu_j\otimes\mu_{j+1,\ldots,k}
\right\|_{\mathrm{TV}}.
\end{equation}

Fix $j\in\{1,\ldots,k-1\}$. Let
$\mu_{j\mid[j+1,k]}(\cdot\mid z)$ denote the conditional measure
of coordinate $j$, given $z=(x_{j+1},\ldots,x_k)$. By
disintegration,
\begin{align*}
\mu_{j,\ldots,k}(dt,dz)
&=
\mu_{j\mid[j+1,k]}(dt\mid z)\,
\mu_{j+1,\ldots,k}(dz),\\
\mu_j
&=
\int\mu_{j\mid[j+1,k]}(\cdot\mid z')
\,\mu_{j+1,\ldots,k}(dz').
\end{align*}
Consequently,
\begin{align*}
&\left\|
\mu_{j,\ldots,k}
-\mu_j\otimes\mu_{j+1,\ldots,k}
\right\|_{\mathrm{TV}}\\
&\quad\le
\iint
\left\|
\mu_{j\mid[j+1,k]}(\cdot\mid z)
-\mu_{j\mid[j+1,k]}(\cdot\mid z')
\right\|_{\mathrm{TV}}
\,\mu_{j+1,\ldots,k}(dz')
\,\mu_{j+1,\ldots,k}(dz).
\end{align*}
Thus, by \eqref{Eq:TelescopingFactorisation}, it suffices to show
that, for every $z,z'\in\T^{k-j}$,
\begin{equation}\label{quasiproducticity}
\left\|
\mu_{j\mid[j+1,k]}(\cdot\mid z)
-\mu_{j\mid[j+1,k]}(\cdot\mid z')
\right\|_{\mathrm{TV}}
\le \frac{5\gamma(k-j)}{8N}.
\end{equation}

Let $\rho$ be the density of $\mu$, and let $p$ be the density of
$\mu_{j,\ldots,k}$. Arguing as in the
proof of Proposition~\ref{Thm:ConcMeasure}, we obtain
\[
\left\|
\mu_{j\mid[j+1,k]}(\cdot\mid z)
-\mu_{j\mid[j+1,k]}(\cdot\mid z')
\right\|_{\mathrm{TV}}
\le
\frac{k-j}{8}
\max_{\ell=j+1,\ldots,k}
\|\partial_{j\ell}\log p\|_\infty.
\]
Thus, to prove \eqref{quasiproducticity}, it is enough to show
\begin{equation}\label{Eq:MarginalMixedBound}
\|\partial_{j\ell}\log p\|_\infty
\le\frac{2\gamma}{N}
\qquad(\ell=j+1,\ldots,k).
\end{equation}
To this end, let $i_1,\ldots,i_m$ be the coordinates outside
$\{j,\ldots,k\}$, and, for $n \le m$ let $p_n$ be the density obtained by
integrating $\rho$ over $i_1,\ldots,i_n$. Thus $p_0=\rho$ and
$p_m=p$.

We claim that, for every $0\le n\le m$,
\begin{equation}\label{Eq:InductiveMixedBound}
\max_{\substack{a\ne b\\a,b\notin\{i_1,\ldots,i_n\}}}
\|\partial_{ab}\log p_n\|_\infty
\le U_n,
\end{equation}
where the sequence $U_n$ satisfies $U_0=\frac{\gamma}{N}$ and 
\[
U_{n+1}=U_n+\frac{U_n^2}{16}.
\]
Before proving the claim, note that it implies
\eqref{Eq:MarginalMixedBound}. Indeed,
\begin{align*}
\frac1{U_{m}}
&=\frac1{U_{m-1}}-\frac1{16+U_{m-1}}\\
&\ge\frac1{U_{m-1}}-\frac1{16}\\&\ge \frac{1}{U_0}-\frac{m}{16}\\
&= \frac{N}{\gamma}-\frac{m}{16},
\end{align*}
and \eqref{Eq:MarginalMixedBound} follows since $m\le N$ and $\gamma \in (0,1).$

It remains to prove \eqref{Eq:InductiveMixedBound}.
The case $n=0$ follows from the definition of
$\mc Y_N(a,M,\gamma)$. Suppose that
\eqref{Eq:InductiveMixedBound} holds for some $n<m$. Writing
$s$ for coordinate $i_{n+1}$ and $x$ for the remaining
coordinates, set
\[
p_{n+1}(x)=\int_\T p_n(x,s)\,ds,
\qquad
w_x(s)=\frac{p_n(x,s)}{p_{n+1}(x)}.
\]
For distinct retained coordinates $a,b$, differentiation gives
\begin{align*}
\partial_{ab}\log p_{n+1}(x)
&=\int_\T\partial_{ab}\log p_n(x,s)\,w_x(s)\,ds\\
&\quad+
\operatorname{Cov}_{w_x(s)\,ds}
\bigl(\partial_a\log p_n(x,\cdot),
      \partial_b\log p_n(x,\cdot)\bigr).
\end{align*}
By the inductive hypothesis, the first term is bounded in
absolute value by $U_n$. Moreover, for $c\in\{a,b\}$,
\[
\operatorname{osc}_s\bigl(\partial_c\log p_n(x,s)\bigr)
\le\frac12\|\partial_{sc}\log p_n\|_\infty
\le\frac{U_n}{2}.
\]
The covariance inequality therefore gives
\[
\|\partial_{ab}\log p_{n+1}\|_\infty
\le U_n+\frac14\left(\frac{U_n}{2}\right)^2 
=U_{n+1},
\]
proving the inductive step. 
\end{proof}

\begin{proof}[Proof of Proposition \ref{Prop:OneStepTV}]
Fix $a,M>0,\gamma \in (0,1)$, $i\in\{1,\ldots,N\}$ and let $\mu \in  \mc Y_N(a,M,\gamma)$. Let $F_i$ be the $i$th coordinate map
of $\bo F$ defined in \eqref{eq:defF}, and let $F_{\mu,i}$ be its
mean-field approximation from Definition~\ref{meanfieldapprox}.
For $x\in\T$, $z\in\T^{N-1}$, and $u\in[0,1]$, write
\[
T_{u,z}(x):=F_{\mu,i}(x)+u\Delta_i(x;z),
\]
where $\Delta_i$ is as in \eqref{deltai}. In particular,
$T_{0,z}=F_{\mu,i}$ and $T_{1,z}=F_i(\cdot;z)$.

Let $\widehat\mu_i:=\widehat\Pi_i\mu$ be the marginal of $\mu$
on all coordinates except $i$, and let $v_z(x)$ denote the
conditional density of coordinate $i$, given that the other
coordinates are equal to $z$. Writing $\mathcal L_T$ for the
transfer operator of a map $T$, we have
\begin{equation}\label{Eq:TVInitialDifference}
\frac{d(\Pi_i\bo F_*\mu)}{d\lambda}
-\frac{d(\Pi_i\bo{\mc F}\mu)}{d\lambda}
=
\int\bigl(\mathcal L_{T_{1,z}}v_z
-\mathcal L_{T_{0,z}}v_z\bigr)\,\widehat\mu_i(dz).
\end{equation}

We claim that this difference can be written as
\begin{align}
\frac{d(\Pi_i\bo F_*\mu)}{d\lambda}
-\frac{d(\Pi_i\bo{\mc F}\mu)}{d\lambda}
={}&-\partial_y\mathcal L_{F_{\mu,i}}m_i
\label{Eq:1DMargCompTV}\\
&+\int_0^1(1-u)\int
\partial_y^2\mathcal L_{T_{u,z}}
\bigl(v_z\Delta_i(\cdot;z)^2\bigr)
\,\widehat\mu_i(dz)\,du,\nonumber
\end{align}
where $
m_i:=\int v_z\Delta_i(\cdot;z)\,\widehat\mu_i(dz). $

To prove the claim, fix $z$ and let $\varphi$ be a smooth periodic
test function. Differentiating with respect to $u$ gives
\[
\begin{aligned}
\frac{d}{du}\int_\T\varphi\,\mathcal L_{T_{u,z}}v_z
&=\frac{d}{du}\int_\T\varphi(T_{u,z}(x))v_z(x)\,dx\\
&=\int_\T\varphi'(T_{u,z}(x))
\Delta_i(x;z)v_z(x)\,dx\\
&=\int_\T\varphi'(y)
\mathcal L_{T_{u,z}}(v_z\Delta_i)(y)\,dy.
\end{aligned}
\]
Integration by parts therefore yields
\[
\frac{d}{du}\mathcal L_{T_{u,z}}v_z
=-\partial_y\mathcal L_{T_{u,z}}(v_z\Delta_i).
\]
Differentiating once more, we obtain
\[
\frac{d^2}{du^2}\mathcal L_{T_{u,z}}v_z
=\partial_y^2\mathcal L_{T_{u,z}}(v_z\Delta_i^2).
\]
Taylor's formula, combined with \eqref{Eq:TVInitialDifference}, proves \eqref{Eq:1DMargCompTV}.

Applying \eqref{firstorder} and \eqref{mathcalr} to $T_{u,z}$, we find that there exists a constant $C$, depending only on $\kappa,K,E$ such that, for all $v\in W^{m.1}$
\begin{align*}
\|\partial_y^m\mathcal L_{T_{u,z}}v\|_1
\le C\|v\|_{W^{m,1}},\qquad m=1,2.
\end{align*}
The above formula, combined with \eqref{Eq:1DMargCompTV}, gives
\begin{align*}
\left\|
\frac{d(\Pi_i\bo F_*\mu)}{d\lambda}
-\frac{d(\Pi_i\bo{\mc F}\mu)}{d\lambda}
\right\|_1
&\le \|\partial_y\mathcal L_{F_{\mu,i}}m_i\|_1\\
&\quad+\int_0^1(1-u)\int
\left\|\partial_y^2\mathcal L_{T_{u,z}}
\bigl(v_z\Delta_i(\cdot;z)^2\bigr)\right\|_1
\,\widehat\mu_i(dz)\,du\\
&\le C\|m_i\|_{W^{1,1}}
+C\int\|v_z\Delta_i(\cdot;z)^2\|_{W^{2,1}}
\,\widehat\mu_i(dz),
\end{align*}
Proposition \ref{Prop:OneStepTV} follows from the following Lemma.
\begin{lemma}\label{appendixlemma}
Let $\mu\in\mc Y_N(a,M,\gamma)$ and fix $i\in\{1,\ldots,N\}$.
Let $\widehat\mu_i$ and $v_z$ be as defined above. Then there is a constant $C$, independent of $N$, $i$, and $\mu$,
such that
\begin{equation}\label{sobolevone}
\|m_i\|_{W^{1,1}}\le \frac{C}{N},
\end{equation}
and
\begin{equation}\label{sobolevtwo}
\int\|v_z\Delta_i(\cdot;z)^2\|_{W^{2,1}}
\,\widehat\mu_i(dz)
\le \frac{C}{N}.
\end{equation}
\end{lemma}

\begin{proof}
Write $p_j$ for the density of $\Pi_j\mu$ and $r_{ij}$ for the
density of $(\pi_i,\pi_j)_*\mu$.
The first estimate in \eqref{sobolevone} follows from
\begin{equation}\label{Eq:PairCOne}
\|r_{ij}-p_ip_j\|_{L^1(\T^2)}
+\|\partial_x(r_{ij}-p_ip_j)\|_{L^1(\T^2)}
\le \frac{C\gamma}{N}.
\end{equation}
Indeed, by the definition
of $\Delta_i$ and $v_z$, we have the equality,
\[
\begin{aligned}
m_i(x)
&=\int v_z(x)\Delta_i(x;z)\,\widehat\mu_i(dz)\\
&=\frac1N\sum_{j\ne i}
\int v_z(x)\left[
h_{ij}(x,z_j)
-\int_\T h_{ij}(x,y)p_j(y)\,dy
\right]\widehat\mu_i(dz)\\
&=\frac1N\sum_{j\ne i}\left[
\int v_z(x)h_{ij}(x,z_j)\,\widehat\mu_i(dz)
-\left(\int v_z(x)\,\widehat\mu_i(dz)\right)
\int_\T h_{ij}(x,y)p_j(y)\,dy
\right]\\
&=\frac1N\sum_{j\ne i}\left[
\int_\T h_{ij}(x,y)r_{ij}(x,y)\,dy
-p_i(x)\int_\T h_{ij}(x,y)p_j(y)\,dy
\right]\\
&=\frac1N\sum_{j\ne i}\int_\T h_{ij}(x,y)
\bigl(r_{ij}(x,y)-p_i(x)p_j(y)\bigr)\,dy,
\end{aligned}
\]
Applying the product rule and using $\|h_{ij}\|_{C^1}\le E$, we obtain
\begin{align*}
\|m_i\|_{W^{1,1}(\T)}
&\le \frac{E}{N}\sum_{j\ne i}
\left(
2\|r_{ij}-p_ip_j\|_{L^1(\T^2)}
+\|\partial_x(r_{ij}-p_ip_j)\|_{L^1(\T^2)}
\right)\\
&\le \frac{CE\gamma(N-1)}{N^2}
\le \frac{CE\gamma}{N},
\end{align*}
where the second inequality follows from \eqref{Eq:PairCOne}. This gives in \eqref{sobolevone}. 

The first term in \eqref{Eq:PairCOne} is bounded by $C\gamma/N$
by applying Lemma~\ref{Lem:S6Marginalisation} to $S=\{i,j\}$. 

For the second inequality in \eqref{Eq:PairCOne}, the chain rule gives
\begin{align*}
\partial_x(r_{ij}-p_ip_j)
&=r_{ij}\,\partial_x\log r_{ij}-p_i'p_j\\
&=r_{ij}\left(\partial_x\log r_{ij}-\frac{p_i'}{p_i}\right)
+\frac{p_i'}{p_i}(r_{ij}-p_ip_j).
\end{align*}

Note that, since $p_i(x)=\int_\T r_{ij}(x,s)\,ds$, we have
\[
\begin{aligned}
\partial_x\log r_{ij}(x,y)-\frac{p_i'(x)}{p_i(x)}
=\int_\T
\bigl[\partial_x\log r_{ij}(x,y)
-\partial_x\log r_{ij}(x,s)\bigr]
\frac{r_{ij}(x,s)}{p_i(x)}\,ds.
\end{aligned}
\]
It follows from \eqref{Eq:MarginalMixedBound}, applied to the
marginal on coordinates $i,j$, that
\[
\left\|\partial_x\log r_{ij}-\frac{p_i'}{p_i}\right\|_\infty
\le \frac12\|\partial_{xy}\log r_{ij}\|_\infty
\le \frac{\gamma}{N}.
\]
Combining the above estimates with the bound on the first term in \eqref{Eq:PairCOne} and
$\|p_i'/p_i\|_\infty\le a$, which follows from
$\mu\in\mc Y_N(a,M,\gamma)$, proves \eqref{Eq:PairCOne}.

To prove \eqref{sobolevtwo}, first observe that, since
$\mu\in\mc Y_N(a,M,\gamma)$, the conditional density $v_z$ satisfies $|v_z'|\le av_z,$ and $|v_z''|\le(M+a^2)v_z.$
Applying the product rule to $v_z\Delta_i(\cdot;z)^2$ and using
$2|AB|\le A^2+B^2$, we obtain
$$\int\|v_z\Delta_i(\cdot;z)^2\|_{W^{2,1}}
\,\widehat\mu_i(dz)
\le C\sum_{r=0}^2
\int\!\int_\T v_z(x)
|\partial_x^r\Delta_i(x;z)|^2\,dx\,\widehat\mu_i(dz)$$
It therefore remains to show that, for $r=0,1,2$,
\begin{equation}\label{integralstep}
\int\!\int_\T v_z(x)
|\partial_x^r\Delta_i(x;z)|^2\,dx\,\widehat\mu_i(dz)
\le \frac{C}{N}.
\end{equation}

For the second estimate, define, for $r=0,1,2$,
\[
z_{ij}^{(r)}(x,y)
:=\partial_1^rh_{ij}(x,y)
-\int_\T\partial_1^rh_{ij}(x,w)p_j(w)\,dw.
\]
Then $|z_{ij}^{(r)}|\le2E$,
$\int_\T z_{ij}^{(r)}(x,y)p_j(y)\,dy=0$, and
\[
\partial_x^r\Delta_i(x;z)
=\frac1N\sum_{j\ne i}z_{ij}^{(r)}(x,z_j).
\]
Using $v_z(x)\,dx\,\widehat\mu_i(dz)=d\mu(x,z)$ and expanding
the square, we obtain
\[
\begin{aligned}
&\int\!\int_\T v_z(x)
|\partial_x^r\Delta_i(x;z)|^2\,dx\,\widehat\mu_i(dz)\\
&\quad=\frac1{N^2}\sum_{j\ne i}
\int |z_{ij}^{(r)}(x,z_j)|^2\,d\mu(x,z)\\
&\qquad+\frac1{N^2}
\sum_{\substack{j,k\ne i\\j\ne k}}
\int z_{ij}^{(r)}(x,z_j)z_{ik}^{(r)}(x,z_k)\,d\mu(x,z).
\end{aligned}
\]
Since $|z_{ij}^{(r)}|\le2E$, the first sum satisfies
\[
\frac1{N^2}\sum_{j\ne i}
\int |z_{ij}^{(r)}(x,z_j)|^2\,d\mu(x,z)
\le\frac{4E^2(N-1)}{N^2}.
\]

To estimate the second sum, fix distinct $j,k\ne i$ and write
$r_{ijk}$ for the density of the marginal of $\mu$ on coordinates
$i,j,k$. Note that
\[
\begin{aligned}
&\int_{\T^3}z_{ij}^{(r)}(x,y)z_{ik}^{(r)}(x,w)
p_i(x)p_j(y)p_k(w)\,dx\,dy\,dw\\
&\qquad=
\int_\T p_i(x)
\left(\int_\T z_{ij}^{(r)}(x,y)p_j(y)\,dy\right)
\left(\int_\T z_{ik}^{(r)}(x,w)p_k(w)\,dw\right)dx
=0.
\end{aligned}
\]
Thus, subtracting this integral and applying
Lemma~\ref{Lem:S6Marginalisation}, we obtain
\[
\begin{aligned}
&\left|\int z_{ij}^{(r)}(x,z_j)z_{ik}^{(r)}(x,z_k)
\,d\mu(x,z)\right|\\
&\quad=\left|\int_{\T^3}
z_{ij}^{(r)}(x,y)z_{ik}^{(r)}(x,w)
\bigl(r_{ijk}(x,y,w)-p_i(x)p_j(y)p_k(w)\bigr)
\,dx\,dy\,dw\right|\\
&\quad\le4E^2\|r_{ijk}-p_ip_jp_k\|_{L^1(\T^3)}
\le\frac{C\gamma E^2}{N}.
\end{aligned}
\]
There are $N-1$ diagonal terms and $(N-1)(N-2)$ ordered
cross terms. Combining their bounds gives, for $r=0,1,2$,
\[
\begin{aligned}
\int\!\int_\T v_z(x)
|\partial_x^r\Delta_i(x;z)|^2\,dx\,\widehat\mu_i(dz)
&\le \frac{4E^2(N-1)}{N^2}
+\frac{C\gamma E^2(N-1)(N-2)}{N^3}\\
&\le \frac{(4+C\gamma)E^2}{N}.
\end{aligned}
\]
This proves \eqref{integralstep}. Substituting these three
bounds into the preceding product-rule estimate yields
\[
\int\|v_z\Delta_i(\cdot;z)^2\|_{W^{2,1}}\,
\widehat\mu_i(dz)
\le \frac{C'}{N},
\]
where $C'$ is independent of $N$, $i$, and $\mu$.
This proves \eqref{sobolevtwo}.
\end{proof}
\end{proof}
% Use the project bibliography when available. The standalone entries reproduce the supplied PDF.
\IfFileExists{Bibilography.bib}{%
  \bibliographystyle{amsalpha}
  \bibliography{Bibilography}
}{%
\providecommand{\bysame}{\leavevmode\hbox to3em{\hrulefill}\thinspace}

}
\end{document}